\documentclass[12pt]{amsart}
\usepackage[english]{babel}
\usepackage{amsmath}
\usepackage{amsthm}
\usepackage{anysize}
\usepackage{mathtools}
\usepackage{mathrsfs}
\usepackage{amssymb}
\usepackage{xcolor}
\usepackage{xfrac}
\usepackage{esint}
\usepackage{graphicx}
\usepackage{float}
\usepackage{array}
\usepackage{enumitem}
\usepackage{tikz-cd}
\usepackage{centernot}
\usepackage{stmaryrd}
\usepackage{comment}
\excludecomment{confidential}

\numberwithin{equation}{section}

\newtheorem{theorem}{Theorem}[section]

\newtheorem{lemma}[theorem]{Lemma}

\newtheorem{definition}[theorem]{Definition}

\DeclareMathOperator{\Alg}{Alg}

\DeclareMathOperator{\ord}{ord}

\DeclareMathOperator{\Tor}{Tor}
\DeclareMathOperator{\Mat}{Mat}

\DeclareMathOperator{\codim}{codim}

\begin{document}
	\title{New analytical and geometrical aspects of the algebraic multiplicity}
\author{Juli\'an L\'opez-G\'omez, Juan Carlos Sampedro} \thanks{The first author has been supported by the Research Grant PGC2018-097104-B-100 of the Spanish Ministry of Science, Technology and Universities and by the Institute of Interdisciplinar Mathematics of Complutense University. The second author has been supported by PhD Grant PRE2019\_1\_0220 of the Basque Country Government.}
	
	\address{Institute of Interdisciplinary Mathematics (IMI) \\
		Department of Analysis and Applied Mathematics \\
		Complutense University of Madrid \\
		28040-Madrid \\
		Spain.}
	\email{Lopez$\_$Gomez@mat.ucm.es, juancsam@ucm.es}



\begin{abstract}
This paper reveals some new analytical and geometrical properties of the generalized algebraic multiplicity, $\chi$,  introduced in \cite{ELG,E} and further developed in \cite{LG01,LGMC,JLJC}. In particular, it establishes a completely new connection  between $\chi$ and the concept of local intersection index of algebraic varieties, a central device in Algebraic Geometry.  This link between Nonlinear Spectral Theory and Algebraic Geometry provides to $\chi$ with a deep geometrical meaning. Moreover, $\chi$ is characterized through the new notion of local determinant of the Schur operator associated to the linear path, $\mathfrak{L}(\lambda)$.
\end{abstract}

\maketitle

\section{Introduction}

\noindent This paper establishes some new  properties of the generalized algebraic multiplicity introduced in \cite{ELG,E} and further developed in \cite{LG01,LGMC,JLJC}, denoted by $\chi$ in this paper. In particular, it reveals a hidden connection  between $\chi$ and the concept of local intersection index of algebraic varieties, a central device in Algebraic Geometry.  $\chi$ is  a concept of multiplicity for general curves of linear Fredholm operators, $\mathfrak{L}(\lambda)$, which has shown to be pivotal in Bifurcation Theory and Nonlinear Analysis.
\par
Classical Spectral Theory deals with special  curves $\mathfrak{L}\in \mathcal{C}(\Omega,\mathcal{L}_c(U,V))$ of the form
$$
  \mathfrak{L}(\lambda)= \lambda I_U -K,\qquad \lambda\in\Omega,
$$
where $U$ and $V$ are two $\mathbb{K}$-Banach spaces such that $U\subset V$, $\mathbb{K}\in\{\mathbb{R},\mathbb{C}\}$,  $\Omega$ is an open domain of $\mathbb{K}$, $\mathcal{L}_{c}(U,V)$ stands for the space of linear continuous operators that are compact perturbations of the identity map, $I_{U}$, and $K$ is a given compact operator. Adopting a geometrical perspective,
Classical Spectral Theory studies the intersections of the straight line $\mathfrak{L}(\lambda)$ with the space of singular operators
$$
  \mathcal{S}(U,V):=\mathcal{L}(U,V)\backslash GL(U,V),
$$
where  $GL(U,V)$ denotes the space of invertible operators. In this context,   $\lambda_{0}\in \Omega$ is said to be an eigenvalue of the straight line $\mathfrak{L}(\lambda)=\lambda I_{U}-K$ if $\mathfrak{L}(\lambda_{0})\in \mathcal{S}(U,V)$, i.e., if $\lambda_0\in\sigma(K)$. Note that, in particular, these linear paths lie in the set of Fredholm operators of index zero, denoted in this paper by $\Phi_0(U,V)$.
\par
More generally, Nonlinear Spectral Theory deals with general continuous paths in $\Phi_{0}(U,V)$,  $\mathfrak{L}\in \mathcal{C}(\Omega,\Phi_{0}(U,V))$, generalizing the classical theory not only because it deals with arbitrary continuous curves, not necessarily straight lines, but also because these paths can lie in  $\Phi_{0}(U,V)$ remaining outside $\mathcal{L}_c(U,V)\varsubsetneq \Phi_{0}(U,V)$.
In this general context, given a Fredholm path, $\mathfrak{L}\in \mathcal{C}(\Omega,\Phi_{0}(U,V))$, a value $\lambda\in\Omega$ is said to be a \emph{generalized eigenvalue} of $\mathfrak{L}$ if $\mathfrak{L}(\lambda)\notin GL(U,V)$. Then, the \emph{generalized spectrum} of $\mathfrak{L}$, $\Sigma(\mathfrak{L})$,  consists of the set of generalized eigenvalues,  i.e., 	
\begin{equation*}
\Sigma(\mathfrak{L}):=\{\lambda\in\Omega: \mathfrak{L}(\lambda)\notin GL(U,V)\}.
\end{equation*}
Thus, as in the classical case, $\Sigma(\mathfrak{L})$ also consists of the intersection points of the curve $\mathfrak{L}$ with the singular manifold  $\mathcal{S}(U,V)\cap\Phi_{0}(U,V)$. Naturally, for every compact operator $K$,
$$
  \Sigma(\lambda I_U-K)=\sigma(K).
$$
An important ingredient in Nonlinear Spectral Theory is the concept of \textit{generalized algebraic multiplicity}, which assigns to every pair $(\mathfrak{L},\lambda_{0})$ with $\mathfrak{L}\in \mathcal{C}^\infty(\Omega,\Phi_{0}(U,V))$ and $\lambda_{0}\in\Sigma(\mathfrak{L})$, an integer
$\chi[\mathfrak{L},\lambda_{0}]\in [1,+\infty]$ such that, whenever  $\dim U = \dim V <\infty$,
\begin{equation}
\label{1.0}
  \chi[\mathfrak{L},\lambda_{0}]=\ord_{\lambda_0}\det\mathfrak{L}(\lambda).
\end{equation}
Although a number of constructions of $\chi[\mathfrak{L},\lambda_{0}]$ had been done
by G\"{o}hberg and Sigal \cite{GS}, Magnus \cite{Ma}, Ize \cite{Iz}, Esquinas and L\'{o}pez-G\'{o}mez \cite{ELG}, Esquinas \cite{Es} and Rabier \cite{Ra1}, it was not until 2001, that Chapters 4 and 5  of L\'{o}pez-G\'{o}mez \cite{LG01} characterized whether all these generalized algebraic multiplicities were finite
through the pivotal concept of \emph{algebraic eigenvalue}.  A generalized eigenvalue,
$\lambda_0\in\Omega$, of a continuous path of Fredholm
operators, $\mathfrak{L}(\lambda)\in \mathcal{C}(\Omega,\Phi_{0}(U,V))$, is said to be $\kappa$\textit{-algebraic} if there exist $\varepsilon>0$ and $C>0$  such that $\mathfrak{L}(\lambda)$ is an isomorphism whenever  $0<|\lambda-\lambda_0|<\varepsilon$, and
\begin{equation}
\label{1.1}
\|\mathfrak{L}^{-1}(\lambda)\|<\frac{C}{|\lambda-\lambda_{0}|^{\kappa}}\quad\hbox{if}\;\;
0<|\lambda-\lambda_0|<\varepsilon,
\end{equation}
with $\kappa$ the minimal integer for which \eqref{1.1} holds.  Throughout this paper, the set of $\kappa$-transversal eigenvalues of $\mathfrak{L}(\lambda)$ is denoted by $\Alg_{\kappa}(\mathfrak{L})$,
 and the set of \emph{algebraic eigenvalues} of $\mathfrak{L}$ is defined by
\[
\Alg(\mathfrak{L}):=\bigcup_{\kappa\in\mathbb{N}}\Alg_\kappa(\mathfrak{L}).
\]
It turns out that the multiplicities of \cite{Ma,ELG,Es} are well defined if, and only if, the path $\mathfrak{L}(\lambda)$ is of class $\mathcal{C}^r$ for some $r\geq 1$ and  $\lambda_0\in \Alg_{\kappa}(\mathfrak{L})$ for some $1\leq \kappa\leq r$.
Moreover, by Theorems 4.4.1 and 4.4.4 of L\'{o}pez-G\'{o}mez \cite{LG01}, when $\mathfrak{L}$ is analytic and $\mathfrak{L}(\Omega)$ contains some invertible operator, then $\Sigma(\mathfrak{L})$ is discrete, and every $\lambda_0\in\Sigma(\mathfrak{L})$  is an algebraic eigenvalue of $\mathfrak{L}$. Thus, through this paper, it is convenient to denote by $\mathcal{A}_{\lambda_{0}}(\Omega_{\lambda_{0}},\Phi_{0}(X,Y))$ the set of curves $\mathfrak{L}\in\mathcal{C}^{r}(\Omega_{\lambda_{0}},\Phi_{0}(X,Y))$ such that $\lambda_{0}\in\Alg_{\kappa}(\mathfrak{L})\cap\Omega_{\lambda_0}$ with $1\leq \kappa \leq r$ for some $r\in\mathbb{N}$. According to \cite[Ch. 4]{LG01},  the generalized algebraic multiplicity $\chi[\mathfrak{L},\lambda_{0}]$ is well defined if and only if $\mathfrak{L}\in\mathcal{A}_{\lambda_{0}}(\Omega_{\lambda_{0}},\Phi_{0}(X,Y))$.
\par
Short time later, in 2004,   the theory of algebraic multiplicities for $\mathcal{C}^\infty$-Fredholm paths was axiomatized by Mora-Corral \cite{MC} by establishing that, modulus a normalization condition (see Theorem \ref{th2.2} for the precise statement), given  an open connected neighborhood of $\lambda_{0}$ in $\mathbb{K}$, denoted by $\Omega_{\lambda_{0}}$,  the algebraic multiplicity $\chi$ is the unique map
\begin{equation*}
\chi[\cdot, \lambda_{0}]: \mathcal{C}^{\infty}(\Omega_{\lambda_{0}}, \Phi_{0}(U))\longrightarrow [0,\infty]
\end{equation*}
satisfying the \emph{product formula}
\begin{equation*}
\chi[\mathfrak{L}\circ\mathfrak{M}, \lambda_{0}]=\chi[\mathfrak{L},\lambda_{0}]+\chi[\mathfrak{M},\lambda_{0}]
\end{equation*}
for all $\mathfrak{L}, \mathfrak{M} \in \mathcal{C}^{\infty}(\Omega_{\lambda_{0}}, \Phi_{0}(U))$,
regardless whether, or not, $\lambda_0$ is singular. These findings were collected by L\'{o}pez-G\'{o}mez and Mora-Corral in \cite{LGMC}, where, in addition, the theory of
G\"{o}hberg and Sigal \cite{GS} was substantially generalized  to a non-holomorphic setting,
and the existence of the local Smith  form was established in the class
$\mathcal{A}_{\lambda_{0}}(\Omega_{\lambda_{0}},\Phi_{0}(X,Y))$ through the lengths of the Jordan chains of $\mathfrak{L}(\lambda)$.
\par
The  relevance of  $\chi[\mathfrak{L},\lambda_{0}]$ in Nonlinear Analysis relies on the crucial fact that, in the special case when $U=V$, $\mathbb{K}=\mathbb{R}$,
$\Omega_{\lambda_0}=(a,b)$ for some $a<b$, $\lambda_0\in (a,b)$ is isolated in $\Sigma(\mathfrak{L})$, and $I_U-\mathfrak{L}(\lambda)$ is compact for all $\lambda\in(a,b)$, by \cite[Th. 5.6.2]{LG01}, $\chi[\mathfrak{L},\lambda_0]$ is odd if and only if the Leray--Schauder degree $\deg_{LS}(\mathfrak{L}(\lambda),B_R(0))$ changes as $\lambda$ crosses $\lambda_0$ (see L\'{o}pez-G\'{o}mez and Sampedro \cite{JLJC} for far more general results in the Fredholm context). In this paper, for any given $u_0\in U$ and $R>0$, we denote by $B_R(u_0)$ the ball of radius $R$ centered at $u_0$. Note that, whenever $\dim U<+\infty$, one can chose a basis in $U$ so that
$$
  \deg_{LS}(\mathfrak{L}(\lambda),B_R(0))=\mathrm{sign\,det}\mathfrak{L}(\lambda)
$$
for all $\lambda \in (a,b)\setminus \Sigma(\mathfrak{L})$.
These changes of degree have relevant consequences in Nonlinear Analysis, because they  entail that,
for every completely continuous  map $\mathfrak{N}(\lambda,u)$ such that $\mathfrak{N}(\lambda,0)=0$ and $\mathfrak{N}(\lambda,u)=o(\|u\|)$ as $u\to 0$, the set of nontrivial solutions of the equation
\begin{equation}
\label{1.2}
\mathscr{F}(\lambda,u)\equiv \mathfrak{L}(\lambda)u+\mathfrak{N}(\lambda,u)=0
\end{equation}
possesses a connected component, denoted by $\mathscr{C}$,  bifurcating from $u=0$ at $\lambda=\lambda_0$ (see \cite{LG01,LG02}).
By a nontrivial solution it is meant any pair $(\lambda,u)\in\mathscr{F}^{-1}(0)$ with $u\neq 0$. Surprisingly, this occurs for every completely continuous nonlinearity $\mathfrak{N}(\lambda,u)$ if and only if $\chi[\mathfrak{L},\lambda_{0}]$ is odd.  And, in such case, the component  $\mathscr{C}$
satisfies the \emph{global alternative} of Rabinowitz \cite{Ra}, as well as the global theorems of Nirenberg \cite{Ni} and Magnus \cite{Ma} (see \cite[Th. 6.3.1]{LG01}). Indeed, when $(a,b)=\mathbb{R}$, either
\begin{enumerate}
\item[(a)] $\mathscr{C}$ is unbounded in $\mathbb{R}\times U$; or

\item[(b)] there exists $\lambda_1\neq \lambda_0$ such that $(\lambda_1,0)\in \overline{\mathscr{C}}$.
\end{enumerate}
The Leray--Schauder degree had been incorporated to the weaponry of Nonlinear Integral Equations by Krasnoselskij \cite{Kr} in order to prove his celebrated local bifurcation theorem, establishing that
if $K:U\to U$ is compact, $\mathfrak{L}(\lambda)=\lambda I_U -K$ for all $\lambda\in (a,b)$, and $\lambda_0\in \sigma(K)$ is an eigenvalue with odd (classical) algebraic multiplicity, then the component $\mathscr{C}$ exists for every completely continuous map $\mathfrak{N}(\lambda,u)$ satisfying the previous requirements.
\par
The first aim of this paper is to identify $\chi[\mathfrak{L},\lambda_0]$ in terms of the \emph{Schur operator} of the path $\mathfrak{L}(\lambda)$. This provides us with a definition of $\chi$ as a natural extension of the finite dimensional concept \eqref{1.0} adopting a direct approach. In this paper, for any given $T\in \Phi_0(U,V)$, we will denote  by $N[T]$ the null space, or kernel, of $T$, and by $R[T]$ the range, or image, of $T$. Assume $\lambda_0\in\Sigma(\mathfrak{L})$ and let $P\in \mathcal{L}(U)$ and $Q\in \mathcal{L}(V)$ be a pair of projections onto $N[\mathfrak{L}(\lambda_0)]$ and $R[\mathfrak{L}(\lambda_0)]$, respectively. Then,
\begin{equation}
\label{1.3}
U  =(I_{U}-P)(U)\oplus N[\mathfrak{L}(\lambda_0)],\qquad V =R[\mathfrak{L}(\lambda_0)]\oplus(I_{V}-Q)(V),
\end{equation}
and $\mathfrak{L}(\lambda)$ can be expressed as a block operator matrix
\begin{equation}
\label{1.4}
\mathfrak{L}(\lambda) =\left(\begin{array}{cc} L_{11}(\lambda) & L_{12}(\lambda) \\[1ex] L_{21}(\lambda) & L_{22}(\lambda) \end{array}\right)
\end{equation}
where
\begin{equation}
\label{1.5}
\begin{array}{ll}
L_{11}(\lambda):=Q\mathfrak{L}(\lambda) (I_{U}-P), & \quad L_{12}(\lambda):=Q\mathfrak{L}(\lambda)P, \\ L_{21}(\lambda):=(I_{V}-Q)\mathfrak{L}(\lambda)(I_{U}-P), & \quad L_{22}(\lambda):=(I_{V}-Q)\mathfrak{L}(\lambda)P.
\end{array}
\end{equation}
In this context, the \textit{Schur operator} of $\mathfrak{L}(\lambda)$ associated to the projection pair $(P,Q)$ can be defined through
\begin{equation*}
\mathscr{S}_{\mathfrak{L}(\lambda),(P,Q)}(\mathfrak{L}(\lambda)): = L_{22}(\lambda)-L_{21}(\lambda)L_{11}^{-1}(\lambda)L_{12}(\lambda),\qquad \lambda \in\Omega_{\lambda_0}.
\end{equation*}
In terms of this operator,  our first result can be stated as follows.

\begin{theorem}
\label{th1.1}
Assume  $\mathfrak{L}\in\mathcal{C}(\Omega,\Phi_{0}(U,V))$ and $\lambda_0\in\Omega\cap\Sigma(\mathfrak{L})$ is an isolated eigenvalue. Then, for every pair
$(P,Q)$ of $\mathfrak{L}(\lambda_0)$-projections,
\begin{equation}
\label{1.6}
	 \mathscr{S}^{-1}_{\mathfrak{L}(\lambda_{0}),(P,Q)}(\mathfrak{L}(\lambda))  =
P\mathfrak{L}^{-1}(\lambda)(I_{V}-Q).
\end{equation}
If, in addition, $\mathfrak{L}\in\mathcal{A}_{\lambda_{0}}(\Omega_{\lambda_{0}},\Phi_{0}(X,Y))$, then
\begin{equation}
\label{1.7}
 \chi[\mathfrak{L},\lambda_{0}]=
 \ord_{\lambda_{0}}\det\mathscr{S}_{\mathfrak{L}(\lambda_{0}),(P,Q)}(\mathfrak{L}(\lambda))=
 \ord_{\lambda_{0}}\det[P\mathfrak{L}^{-1}(\lambda)(I_{V}-Q)]^{-1}.
\end{equation}
\end{theorem}

Theorem \ref{th1.1} allows us to introduce, in a rather natural manner,  a local concept of infinite-dimensional determinant, $\mathcal{D}_{\mathfrak{L}(\lambda_{0}),(P,Q)}$, for which
$$
\chi[\mathfrak{L},\lambda_{0}]=\ord_{\lambda_{0}}
 \mathcal{D}_{\mathfrak{L}(\lambda_{0}),(P,Q)}(\mathfrak{L}(\lambda))=
 \ord_{\lambda_{0}}\det\mathscr{S}_{\mathfrak{L}(\lambda_{0}),(P,Q)}(\mathfrak{L}(\lambda)).
$$
Indeed, we will see that taking
$$
  \mathcal{D}_{\mathfrak{L}(\lambda_{0}),(P,Q)}(\mathfrak{L}(\lambda)):= \det(L_{11}(\lambda))\cdot\det(\mathscr{S}_{T,(P,Q)}(\mathfrak{L}(\lambda))
$$
is consistent with the classical theory of Schur complements. As a byproduct of Theorem \ref{th1.1}, the multiplicity of Ize \cite{Iz},  introduced for analytic paths, can be defined \emph{mutatis mutandis} for every admissible path $\mathfrak{L}\in\mathcal{A}_{\lambda_{0}}(\Omega_{\lambda_{0}},\Phi_{0}(X,Y))$.
Our next result establishes a sharp connection between these concepts and Bifurcation Theory.

\begin{theorem}
\label{th1.2}
Assume $\mathfrak{L}\in \mathcal{C}^1(\Omega,\Phi_{0}(U,V))$ and $\lambda_0\in\Omega$. Then,
for every pair $(P,Q)$ of $\mathfrak{L}(\lambda_0)$-projections, the Schur operator $\mathscr{S}_{\mathfrak{L}(\lambda_{0}),(P,Q)}(\mathfrak{L}(\lambda))$ coincides with the linearization at $(\lambda,u)=(\lambda,0)$ of the bifurcation equation of \eqref{1.2} through the Lyapunov--Schmidt decomposition of projections $(P,Q)$, derived as discussed in \cite[Ch. 3]{LG01}, which is  denoted by $\mathscr{B}(\lambda)$ in this paper (see Section 4), i.e.
\begin{equation}
\label{1.8}
  \mathscr{B}(\lambda)=\mathscr{S}_{\mathfrak{L}(\lambda_0),(P,Q)} (\mathfrak{L}(\lambda))=
  L_{22}(\lambda) - L_{21}(\lambda)   L_{11}^{-1}(\lambda)L_{12}(\lambda).
\end{equation}
If, in addition, $\mathfrak{L}\in\mathcal{A}_{\lambda_{0}}(\Omega_{\lambda_{0}},\Phi_{0}(X,Y))$, then, thanks to Theorem \ref{th1.1},
\begin{equation}
\label{1.9}
 \chi[\mathfrak{L},\lambda_{0}]=\ord_{\lambda_{0}}\det\mathscr{S}_{\mathfrak{L}(\lambda_{0}),(P,Q)}
 (\mathfrak{L}(\lambda))=\ord_{\lambda_{0}}\det\mathscr{B}(\lambda),
\end{equation}
regardless the pair of projections chosen, $(P,Q)$. Therefore, thanks to Theorem
\cite[Th. 4.3.4]{LG01}, the following assertions are equivalent:
\begin{enumerate}
\item[{\rm (a)}] $\lambda_0$ is a nonlinear eigenvalue of $\mathfrak{L}(\lambda)$, in the sense that
$(\lambda_0,0)$ is a bifurcation point of \eqref{1.2} from $(\lambda,u)=(\lambda,0)$ for all map $\mathfrak{N}(\lambda,u)$ of class $\mathcal{C}^1$.
\item[{\rm (b)}]  $\chi[\mathfrak{L},\lambda_0]$ is an odd integer.
\item[{\rm (c)}] $\det \mathscr{B}(\lambda)$ changes of sign as $\lambda$ crosses $\lambda_0$, regardless $(P,Q)$.
\item[{\rm (d)}] $\det \mathscr{S}_{\mathfrak{L}(\lambda),(P,Q)}(\mathfrak{L}(\lambda))$ changes of sign as $\lambda$ crosses $\lambda_0$, regardless $(P,Q)$.
\item[{\rm (e)}] $\det [P\mathfrak{L}^{-1}(\lambda)(I_{V}-Q)]$ changes of sign as $\lambda$ crosses
$\lambda_0$, regardless $(P,Q)$.
\end{enumerate}
Furthermore, when any of these conditions occurs, there is a component, $\mathscr{C}$,  of the set of nontrivial solutions of \eqref{iv.1}  such that $(\lambda_0,0)\in \overline{\mathscr{C}}$. By a component, it is meant any closed and connected subset maximal for the inclusion.
\end{theorem}

Our second aim is incorporating a completely new perspective into the theory of algebraic multiplicities in order to establish a bisociation between Spectral Theory and Algebraic Geometry.  As already discussed above, any generalized eigenvalue $\lambda_{0}\in\Sigma(\mathfrak{L})$ of the curve $\mathfrak{L}\in\mathcal{C}(\Omega,\Phi_{0}(U,V))$ lies in the intersection of $\mathfrak{L}$ with the singular space $\mathcal{S}(U,V)\cap\Phi_{0}(U,V)$. This geometrical feature leads us to think that the algebraic multiplicity might actually measure  how the curve $\mathfrak{L}(\lambda)$ intersects geometrically with $\mathcal{S}(U,V)\cap\Phi_{0}(U,V)$. Thus, it might be  related to the \textit{local intersection index}, a pivotal geometrical device for measuring the nature of the intersections of varieties.
\par
Given a family of algebraic varieties, $\mathscr{V}$, an intersection theory over $\mathscr{V}$ consists  of giving a pairing
\begin{equation}
\label{1.10}
   \bullet\,:\; A^{r}(X)\times A^{s}(X)\to A^{r+s}(X)
\end{equation}
satisfying a series of axioms (see, e.g.,  Hartshorne \cite[pp. 426--427]{H}, Eisenbud and Harris \cite[Ch. 1 \& 2]{E}, and Fulton \cite[Ch. 7 \& 8]{F}) for every $r,s \in \mathbb{N}$ and $X\in\mathscr{V}$, where $A^{r}(X)$ stands for the group of cycles of codimension $r$ on $X$ modulo rational equivalence. The graded group
$$
 A(X)\equiv \bigoplus_{r\in\mathbb{N}}A^{r}(X)
$$
is referred to as the \textit{Chow group} after Chow \cite{Ch}. The pairing \eqref{1.10} gives to  $A(X)$ the structure of a graded ring, the \textit{Chow ring} of $X$. Giving an intersection theory to an algebraic variety $X$  consists in giving the structure of the Chow ring $A(X)$; the axioms of the pairing $\bullet$ try to mimic in $X$ the celebrated \textit{Bezout Theorem} \cite[\& 2.1.1]{E}. This explains why one of these axioms establishes that if $X_1$ and $X_2$ are subvarieties of $X$
with \emph{proper intersection}, in the sense that any irreducible component of $X_1\cap X_2$ has codimension $\codim X_1 + \codim X_2$, then
$$
[X_1]\bullet [X_2]=\sum_{j} i(X_1,X_2;C_{j})[C_{j}]
$$
where the sum runs over the set of all irreducible components, $C_{j}$, of $X_1\cap X_2$, and  the integer $i(X_1,X_2;C_{j})$ stands for the \textit{local intersection index} of $X_1$ and $X_2$ along $C_{j}$. By a result of Serre \cite[Ch. V, \& C.1]{S}, for any given pair $X_1, X_2$ of subvarieties of a smooth variety $X$ with proper intersection and any irreducible component, $C$,  of $X_1\cap X_2$, the local intersection index of $X_1$ and $X_2$ along $C$ is given through
$$
i(X_1,X_2;C)=\sum_{i=0}^{\dim X}(-1)^{i}\ell_{\mathcal{O}_{c,X}} \Tor^{\mathcal{O}_{c,X}}_{i}(\mathcal{O}_{c,X}/\mathfrak{P}_{X_1},\mathcal{O}_{c,X}/\mathfrak{P}_{X_2})
$$
where $\mathcal{O}_{c,X}$ is the local ring of $c\in C$ in $X$, and $\mathfrak{P}_{X_1}$ and $\mathfrak{P}_{X_2}$ are the ideals  of $X_1$ and $X_2$, respectively,  in the ring $\mathcal{O}_{c,X}$. In this paper, and, in particular, in the Serre formula, we are denoting by $\ell_{R}(M)$ the length of the module $M$ over the ring $R$, and by $\Tor_{i}^{R}$  the $i$-th Tor $R$-module. As in our context  $X$ is a smooth algebraic variety over an algebraically closed field,  $\mathbb{C}$, and $X_1, X_2$
are two irreducible Cohen--Macaulay subvarieties of $X$ such that $X_1\cap X_2=\{x\}$, with proper intersection, according to, e.g., Eisenbaud and Harris  \cite[p. 48]{E}, the local intersection index of $X_1$ and $X_2$ at $x$ reduces to
\begin{equation}
\label{1.11}
i(X_1,X_2;x):=\ell_{\mathcal{O}_{x,X}}(\mathcal{O}_{x,X}/(\mathfrak{P}_{X_1}+\mathfrak{P}_{X_2})).
\end{equation}
The third finding of this paper establishes a connection between the concepts of classical algebraic multiplicity and the intersection index. Subsequently, $\mathcal{D} : \mathcal{L}(\mathbb{C}^{N})\to\mathbb{C}$ stands for the determinant map defined by
$$
  \mathcal{D}(T):=\det T\quad \hbox{for every}\;\; T\in \mathcal{L}(\mathbb{C}^{N}).
$$

\begin{theorem}
\label{th1.3}
Let $T\in\mathcal{L}(\mathbb{C}^{N})$ and $\lambda_{0}\in\sigma(T)$ be, and set  $\mathfrak{L}(\lambda)=\lambda I_{N}-T$, $\lambda\in \mathbb{C}$. Then,  	
\begin{equation}
\label{1.12}
 \chi[\mathfrak{L},\lambda_{0}]=
 i(\mathcal{D}^{-1}(0),\mathfrak{L}(\mathbb{C});\mathfrak{L}(\lambda_{0})).
\end{equation}
\end{theorem}

Finally, in order to state our last finding, for any given $\mathfrak{L}\in \mathcal{A}_{\lambda_{0}}(\Omega_{\lambda_{0}},\Phi_{0}(X,Y))$,  we will denote by
$$
  \mathscr{L}= \mathscr{L}(\mathfrak{L}) \in \mathcal{L}(\mathbb{C}^M)
$$
the \emph{global linearization}, as discussed by Lemma 10.1.1
of \cite{LGMC}, of the \emph{local Smith form} of the \textit{Schur complement} of $\mathfrak{L}(\lambda)$, whose existence follows from Theorem 7.4.1 of L\'{o}pez-G\'{o}mez and Mora-Corral  \cite{LGMC}. The following result is a substantial, very sharp,  generalization of Theorem \ref{th1.3}.

\begin{theorem}\label{th1.4} For every $\mathfrak{L}\in \mathcal{A}_{\lambda_{0}}(\Omega_{\lambda_{0}},\Phi_{0}(X,Y))$,
\begin{equation}
\label{1.13}
\chi[\mathfrak{L},\lambda_{0}]=i(\mathcal{D}^{-1}(0),\lambda I_{M} - \mathscr{L};\lambda_{0} I_{M}-\mathscr{L}).
\end{equation}
\end{theorem}

This paper is distributed as follows. Section 2 reviews briefly the generalized algebraic multiplicity $\chi[\mathfrak{L},\lambda_0]$. In Section 3 we introduce a local notion of determinant for a general path $\mathfrak{L}\in\mathcal{C}(\Omega,\Phi_0(U,V))$ through the Schur operator
$\mathscr{S}_{\mathfrak{L}(\lambda),(P,Q)}(\mathfrak{L}(\lambda))$ associated to a projection pair
$(P,Q)$. Then, we use this concept of infinite-dimensional determinant to prove Theorem \ref{th1.1}.  Section 4 consists of the proof of Theorem \ref{th1.2}, Section 5 shows Theorem \ref{th1.3} and, finally, Section 6 proves Theorem \ref{th1.4}.  In reading this paper, it would be appropriate to have copies of  \cite{LG01} and \cite{LGMC}.

\section{Nonlinear spectral theory. Generalized algebraic multiplicity.}

\noindent Throughout this section, $\mathfrak{L}\in \mathcal{C}(\Omega,\Phi_{0}(U,V))$, where $\Omega$ is a subdomain of $\mathbb{K}\in\{\mathbb{R},\mathbb{C}\}$. According to Theorems 4.4.1 and 4.4.4 of \cite{LG01}, when $\mathfrak{L}(\lambda)$ is analytic in $\Omega$, i.e., $\mathfrak{L}\in\mathcal{H}(\Omega, \Phi_{0}(U,V))$, either $\Sigma(\mathfrak{L})=\Omega$, or $\Sigma(\mathfrak{L})$ is discrete and $\Sigma(\mathfrak{L})\subset \Alg(\mathfrak{L})$.
\par
The next concept, going back to Esquinas and L\'{o}pez-G\'{o}mez \cite{ELG},  is pivotal in Nonlinear Spectral Theory as it allows to introduce a generalized algebraic multiplicity, $\chi[\mathfrak{L},\lambda_0]$, in a rather natural manner. Subsequently, we will denote
$$
\mathfrak{L}_{j}:=\frac{1}{j!}\mathfrak{L}^{(j)}(\lambda_{0}), \quad 0\leq j\leq r,
$$
if these derivatives exist.

\begin{definition}
	\label{de2.1}
	Let $\mathfrak{L}\in \mathcal{C}^{r}(\Omega,\Phi_{0}(U,V))$ and $1\leq \kappa \leq r$. Then, a given $\lambda_{0}\in \Sigma(\mathfrak{L})$ is said to be a $\kappa$-transversal eigenvalue of $\mathfrak{L}$ if
	\begin{equation*}
	\bigoplus_{j=1}^{\kappa}\mathfrak{L}_{j}\left(\bigcap_{i=0}^{j-1}N[\mathfrak{L}_{i}]\right)
	\oplus R(\mathfrak{L}_{0})=V\;\; \hbox{with}\;\; \mathfrak{L}_{\kappa}\left(\bigcap_{i=0}^{\kappa-1}N[\mathfrak{L}_{i}]\right)\neq \{0\}.
	\end{equation*}
\end{definition}

For these eigenvalues, the \emph{algebraic multiplicity of $\mathfrak{L}$ at $\lambda_{0}$}, $\chi[\mathfrak{L},\lambda_0]$,  was introduced in \cite{ELG} through
\begin{equation*}
\chi[\mathfrak{L}, \lambda_{0}] :=\sum_{j=1}^{\kappa}j\cdot \dim \mathfrak{L}_{j}\left(\bigcap_{i=0}^{j-1}N[\mathfrak{L}_{i}]\right).
\end{equation*}
According to Theorems 4.3.2 and 5.3.3 of  \cite{LG01}, for every $\mathfrak{L}\in \mathcal{A}_{\lambda_{0}}(\Omega,\Phi_{0}(U,V))$ there exists a polynomial $\Phi: \mathbb{K}\to \mathcal{L}(U)$ with $\Phi(\lambda_{0})=I_{U}$ and an integer $\kappa\in\{1,2,...,r\}$ such that $\lambda_{0}$ is a $\kappa$-transversal eigenvalue of the path
\begin{equation}
\label{2.1}
\mathfrak{L}^{\Phi}:=\mathfrak{L}\circ\Phi\in \mathcal{C}^{r}(\Omega, \Phi_{0}(U,V)).
\end{equation}
Moreover, $\kappa$ and $\chi[\mathfrak{L}^{\Phi},\lambda_{0}]$ are independent of the curve of \emph{trasversalizing local isomorphisms} $\Phi$ chosen to transversalize $\mathfrak{L}$ at $\lambda_0$ through \eqref{2.1}, regardless $\Phi$ is a polynomial or not. Therefore, the next generalized concept of algebraic multiplicity, defined for every $\mathfrak{L}\in \mathcal{A}_{\lambda_{0}}(\Omega,\Phi_{0}(U,V))$, is consistent
\[
\chi[\mathfrak{L},\lambda_0]:= \chi[\mathfrak{L}^\Phi,\lambda_0].
\]
This concept can be extended by setting
\[
\chi[\mathfrak{L},\lambda_0] :=0 \quad \hbox{if}\;\; \lambda_0\notin\Sigma(\mathfrak{L})
\]
and
\begin{equation}
\label{2.2}
\chi[\mathfrak{L},\lambda_0]: =+\infty \quad \hbox{if}\;\; \lambda_0\in \Sigma(\mathfrak{L})
\setminus \Alg(\mathfrak{L}) \;\; \hbox{and}\;\; r=+\infty,
\end{equation}
so that $\chi[\mathfrak{L},\lambda]$ is well defined for all $\mathfrak{L}\in\mathcal{A}_{\lambda_{0}}(\Omega,\Phi_{0}(U,V))$ at any $\lambda\in \Omega$, and, in particular, for every analytic curve  $\mathfrak{L}\in\mathcal{H}(\Omega,\Phi_{0}(U,V))$.  Setting \eqref{2.2}, the algebraic multiplicity is also well defined for each smooth curve $\mathfrak{L}\in\mathcal{C}^{\infty}(\Omega,\Phi_{0}(U,V))$ at any $\lambda\in\Omega$, and it coincides with the multiplicity of Magnus \cite{Ma}. The next uniqueness result, going back to
Mora-Corral \cite{MC}, axiomatizes the concept of Algebraic Multiplicity (see also Chapter 6 of \cite{LGMC}).

\begin{theorem}
	\label{th2.2}
For every $\lambda_{0}\in\mathbb{K}$ and any open neighborhood $\Omega_{\lambda_{0}}\subset\mathbb{K}$ of $\lambda_{0}$, the algebraic multiplicity $\chi$ is the unique map 		
	\begin{equation*}
	\chi[\cdot, \lambda_{0}]: \mathcal{C}^{\infty}(\Omega_{\lambda_{0}}, \Phi_{0}(U))\longrightarrow [0,\infty]
	\end{equation*}
	such that
	\begin{enumerate}
		\item[{\rm (PF)}] For every pair $\mathfrak{L}, \mathfrak{M} \in \mathcal{C}^{\infty}(\Omega_{\lambda_{0}}, \Phi_{0}(U))$,
		\begin{equation*}
		\chi[\mathfrak{L}\circ\mathfrak{M}, \lambda_{0}]=\chi[\mathfrak{L},\lambda_{0}]+\chi[\mathfrak{M},\lambda_{0}].
		\end{equation*}
		\item[{\rm (NP)}] There exists a rank one projection $\Pi \in \mathcal{L}(U)$ such that
		\begin{equation*}
		\chi[(\lambda-\lambda_{0})\Pi +I_{U}-\Pi,\lambda_{0}]=1.
		\end{equation*}
	\end{enumerate}
\end{theorem}

The axiom (PF) is the  \emph{product formula} and (NP) is a \emph{normalization property}
for establishing the uniqueness of $\chi$. From these two axioms one can derive the remaining properties of  $\chi$; among them, that it equals the classical algebraic multiplicity when
\begin{equation}
\label{2.3}
\mathfrak{L}(\lambda)= \lambda I_{U} - K
\end{equation}
for some compact operator $K$. Indeed, according to \cite{LGMC}, for every $\mathfrak{L}\in \mathcal{C}^{\infty}(\Omega_{\lambda_{0}},\Phi_{0}(U))$, the following properties are satisfied:
\begin{itemize}
	\item $\chi[\mathfrak{L},\lambda_{0}]\in\mathbb{N}\uplus\{+\infty\}$;
	\item $\chi[\mathfrak{L},\lambda_{0}]=0$ if and only if $\mathfrak{L}(\lambda_0)
	\in GL(U)$;
	\item $\chi[\mathfrak{L},\lambda_{0}]<\infty$ if and only if $\lambda_0 \in\Alg(\mathfrak{L})$.
	\item If $U =\mathbb{K}^N$, then, in any basis,
\begin{equation}
\label{2.4}
	\chi[\mathfrak{L},\lambda_{0}]= \mathrm{ord}_{\lambda_{0}}\det \mathfrak{L}(\lambda).
\end{equation}
	\item Let $L\in \mathcal{L}(U)$ be such that $\lambda I_U-L\in \Phi_0(U)$. Then, for each eigenvalue $\lambda_0\in \sigma(L)$, there exists $\kappa\geq 1$ such that
	\begin{equation*}
	\begin{split}
	\chi [\lambda I_U-L,\lambda_{0}]  =\underset{n\in\mathbb{N}}{\sup} \dim  N[(\lambda_{0}I_{U}-L)^{n}]  = \dim N[(\lambda_{0}I_{U}-L)^{\kappa}].
	\end{split}
	\end{equation*}
	\item Let $V$ be another complex Banach space and $\mathfrak{P}\in\mathcal{C}^{\infty}(\Omega_{\lambda_{0}},\Phi_{0}(V))$. Then,
	\begin{equation*} \chi[\mathfrak{L}\oplus\mathfrak{P},\lambda_{0}]=\chi[\mathfrak{L},\lambda_{0}]+
\chi[\mathfrak{P},\lambda_{0}].
	\end{equation*}
\end{itemize}
Therefore, $\chi$ extends, very substantially, the classical concept of algebraic multiplicity, which is only valid for operator paths of the form \eqref{2.3}.

\section{Schur complement and local determinant}

\noindent This section gives a new  equivalent definition of $\chi$,  without invoking the concept of transversalization, based on the formulation of the finite-dimensional algebraic multiplicity in terms of the order of the determinant. To accomplish this task, we need to develop an analogue of the notion of \textit{determinant} in an infinite-dimensional setting. Although the notion of determinant does not exists in such setting, we will see that it can be introduced in local neighborhoods of $\Phi_{0}(U,V)$.
\par
Throughout this section, $T\in \Phi_0(U,V)$ is an arbitrary linear Fredholm operator of index zero,
and $P\in \mathcal{L}(U)$, $Q\in \mathcal{L}(V)$ are projections onto $N[T]$ and $R[T]$, respectively; The pair $(P,Q)$ will be called a pair of $T$-projections. Then,
the following topological direct sum decompositions hold
\begin{equation}
\label{3.1}
U  =(I_{U}-P)(U)\oplus N[T],\qquad V =R[T]\oplus(I_{V}-Q)(V),
\end{equation}
and, setting
\begin{align*}
     R[I_U-P]&=(I_{U}-P)(U)\equiv N[T]^{\perp},\\ R[I_V-Q] & =(I_{V}-Q)(V)\equiv R[T]^{\perp},
\end{align*}
by Fitzpatrick and Pejsachowicz \cite[p. 286]{FP1}, for any given $L\in \Phi_{0}(U,V)$,  $L$ can be equivalently expressed as a block operator matrix
\begin{equation}
\label{3.2}
L=\left(\begin{array}{cc} L_{11} & L_{12} \\[1ex] L_{21} & L_{22} \end{array}\right)
\end{equation}
where
\begin{equation*}
\begin{array}{ll}
L_{11}:=QL(I_{U}-P), & \quad L_{12}:=QLP, \\ L_{21}:=(I_{V}-Q)L(I_{U}-P), & \quad L_{22}:=(I_{V}-Q)LP.
\end{array}
\end{equation*}
In particular, since $TP=0$ and $(I_V-Q)T=0$, the operator $T$ can be expressed as
\begin{equation*}
T=\left(\begin{array}{cc}
T_{11} & 0 \\[1ex]
0 & 0
\end{array}\right)
\end{equation*}
with $T_{11}\in GL(N[T]^{\perp}, R[T])$. Since $GL(N[T]^{\perp},R[T])$ is open in $\mathcal{L}(N[T]^{\perp},R[T])$ and $\Phi_{0}(U,V)$ is open in $\mathcal{L}(U,V)$,  there exists $\varepsilon>0$ such that, whenever $L\in \mathcal{L}(U,V)$ satisfies  $\|T-L\|<\varepsilon$, then
$L\in\Phi_{0}(U,V)$, and it can be expressed as \eqref{3.2} with  $L_{11}\in GL(N[T]^{\perp},R[T])$.
\par
In this context, the \textit{Schur operator} of $T$ associated to the projection pair $(P,Q)$ can be defined through
\begin{equation*}
\begin{array}{lccl}
\mathscr{S}_{T,(P,Q)}: & B_{\varepsilon}(T)\subset\Phi_{0}(U,V) & \longrightarrow & \mathcal{L}(N[T],R[T]^{\perp}) \\ & L & \mapsto & L_{22}-L_{21}L_{11}^{-1}L_{12}
\end{array}
\end{equation*}
where $B_{\varepsilon}(T)$ stands for the open ball of radius $\varepsilon$ centered at $T\in \Phi_0(U,V)$ in $\mathcal{L}(U,V)$. We call this map \textit{Schur operator} by analogy, in the  Euclidean space, to the \textit{Schur complement} of a matrix. Indeed, given any block matrix
\begin{equation}
\label{3.3}
M=\left(\begin{array}{cc} A & B \\[1ex] C & D \end{array}\right)
\end{equation}
with $A\in GL(\mathbb{K}^{n})$, $B\in\Mat_{n\times m} (\mathbb{K})$, $C\in\Mat_{m\times n}(\mathbb{K})$ and $D\in \Mat_{m}(\mathbb{K})$, the \textit{Schur complement} of the matrix $D$ in $M$  is the matrix $M/A\in\Mat_{m}(\mathbb{K})$ defined by
\begin{equation*}
M/ A:=D-CA^{-1}B.
\end{equation*}
By  a lemma of Banachiewicz \cite[p. 50]{B}, the Schur complement satisfies the identity
\begin{equation*}
\det(M)=\det(A)\cdot\det(M/ A)
\end{equation*}
for every block matrix \eqref{3.3}. Based on this feature, we introduce the following local notion of determinant.

\begin{definition}
\label{de3.1}
Let $T\in\Phi_{0}(U,V)$ and $(P,Q)$ be a pair of  $T$-projections. Then, for sufficiently small $\varepsilon>0$, we define the local determinant functional by
\begin{equation*}
	\begin{array}{lccl}
	\mathcal{D}_{T,(P,Q)}: & B_{\varepsilon}(T)\subset \Phi_{0}(U,V) & \longrightarrow & \mathbb{K} \\	& L & \mapsto & \det(L_{11})\cdot\det(\mathscr{S}_{T,(P,Q)}(L)).	\end{array}
\end{equation*}
\end{definition}

\noindent The next result establishes that  the local determinant functional indeed behaves as a (local) determinant on $B_{\varepsilon}(T)$.

\begin{theorem}
\label{th3.2}
Let $T\in\Phi_{0}(U,V)$ and $(P,Q)$ be a pair of $T$-projections. Then, for sufficiently small $\epsilon>0$ and every $L\in B_{\varepsilon}(T)$,
$$
  L\in GL(U,V) \text{ if and only if }  \mathcal{D}_{T,(P,Q)}(L)\neq0.
$$
\end{theorem}

\begin{proof}
Let $L\in B_{\varepsilon}(T)$ be. Then, the Schur operator of $L$ is well defined. Moreover, $L$ can be decomposed as 	
\begin{equation}
\label{3.4}
	L=\left(\begin{array}{cc}
	I_{R[T]} & 0 \\[1ex]
	L_{21}L^{-1}_{11} & I_{R[T]^{\perp}}
	\end{array}\right)\left(\begin{array}{cc}
	L_{11}& 0 \\[1ex]
	0 & \mathscr{S}_{T,(P,Q)}(L)
	\end{array}\right)\left(\begin{array}{cc}
	I_{N[T]^{\perp}} & L^{-1}_{11}L_{12} \\[1ex]
	0 & I_{N[T]}
	\end{array}\right).
\end{equation}
Clearly, the first and the third matrix operators of \eqref{3.4} are invertible. Thus,
taking inverses, we find that  	
\begin{equation}
\label{3.5}
\begin{split}
	L_{11}\oplus\mathscr{S}_{T,(P,Q)}(L) & =\left(\begin{array}{cc}
	L_{11} & 0 \\[1ex]
	0 & \mathscr{S}_{T,(P,Q)}(L)
	\end{array}\right)\\ & = \left(\begin{array}{cc}
	I_{R[T]} & 0 \\[1ex]
	L_{21}L^{-1}_{11} & I_{R[T]^{\perp}}
	\end{array}\right)^{-1} L\left(\begin{array}{cc}
	I_{N[T]^{\perp}} & L^{-1}_{11}L_{12} \\[1ex]
	0 & I_{N[T]} 	\end{array}\right)^{-1}.
\end{split}
\end{equation}
If $L\in GL(U,V)$, then, due to \eqref{3.5},  $L_{11}\oplus\mathscr{S}_{T,(P,Q)}\in GL(U,V)$ because it is a composition of invertible operators. So, since $L_{11}\in  GL(N[T]^{\perp},R[T])$, necessarily
$$
  \mathscr{S}_{T,(P,Q)}(L)\in GL(N[T],R[T]^{\perp})
$$
and therefore $\mathcal{D}_{T,(P,Q)}(L)\neq0$. To prove the converse, assume that  $\mathcal{D}_{T,(P,Q)}(L)=0$. Then,
$$
  \mathscr{S}_{T,(P,Q)}(L)\notin GL(N[T],R[T]^{\perp})
$$
and hence,
$$
   L_{11}\oplus\mathscr{S}_{T,(P,Q)}(L)\notin GL(U,V).
$$
Consequently, by \eqref{3.5}, $L\notin GL(U,V)$. This concludes the proof.
\end{proof}

Thanks to this (local)  notion of determinant, the generalized algebraic multiplicity can be also expressed in the vein of \eqref{2.4} even in an infinite-dimensional setting.
Precisely, the next result holds.

\begin{theorem}
\label{th3.3}
Assume  $\mathfrak{L}\in\mathcal{A}_{\lambda_{0}}(\Omega_{\lambda_{0}},\Phi_{0}(X,Y))$, i.e.
$\mathfrak{L}\in\mathcal{C}^{r}(\Omega,\Phi_{0}(U,V)$ with  $\lambda_{0}\in\Alg_{\kappa}(\mathfrak{L})$ for some $1\leq \kappa \leq r$ and $r\in\mathbb{N}$. Then,
\begin{equation}
\label{3.6}
 \chi[\mathfrak{L},\lambda_{0}]=\ord_{\lambda_{0}}
 \mathcal{D}_{\mathfrak{L}(\lambda_{0}),(P,Q)}(\mathfrak{L}(\lambda))=
 \ord_{\lambda_{0}}\det\mathscr{S}_{\mathfrak{L}(\lambda_{0}),(P,Q)}(\mathfrak{L}(\lambda))
\end{equation}
for every pair $(P,Q)$ of $\mathfrak{L}(\lambda_{0})$-projections.
\end{theorem}
\begin{proof}
First, we will show that, for every pair $(P,Q)$ of $\mathfrak{L}(\lambda_{0})$-projections,
\begin{equation}
\label{3.7}
   \lambda_{0}\in\Alg_{\kappa}(\mathscr{S}_{\mathfrak{L}(\lambda_{0}),(P,Q)}\circ\mathfrak{L}).
\end{equation}
Indeed, let $(P,Q)$ be a pair of $\mathfrak{L}(\lambda_{0})$-projections. Then,
\begin{equation*}
	\mathfrak{L}(\lambda)=\left(\begin{array}{cc}
	L_{11}(\lambda) & L_{12}(\lambda) \\[1ex] 	L_{21}(\lambda) & L_{22}(\lambda) 	\end{array}\right),
\end{equation*}
where
\begin{equation*}
\begin{array}{ll}
L_{11}(\lambda):=Q\mathfrak{L}(\lambda) (I_{U}-P), & L_{12}(\lambda):=Q\mathfrak{L}(\lambda)P, \\[1ex] L_{21}(\lambda):=(I_{V}-Q)\mathfrak{L}(\lambda)(I_{U}-P), & L_{22}(\lambda):=(I_{V}-Q)\mathfrak{L}(\lambda)P.
\end{array}
\end{equation*}
Since
$$
   L_{11}(\lambda_0)\in GL(N[\mathfrak{L}(\lambda_0)]^{\perp},R[\mathfrak{L}(\lambda_0)]),
$$
we can take $\Omega_{\lambda_0}$ sufficiently small so that
$$
   L_{11}(\lambda)\in GL(N[\mathfrak{L}(\lambda_0)]^{\perp},
   R[\mathfrak{L}(\lambda_0)])\quad \hbox{for all}\;\; \lambda\in \Omega_{\lambda_0}.
$$
Thus, by \eqref{3.5}, it becomes apparent that
\begin{equation}
\label{3.8}	
L_{11}(\lambda)\oplus\mathscr{S}_{\mathfrak{L}(\lambda_{0}),(P,Q)}(\mathfrak{L}(\lambda)) =  \mathfrak{L}_{1}^{-1}(\lambda)\mathfrak{L}(\lambda)\mathfrak{L}^{-1}_{2}(\lambda)
\quad \hbox{for all}\;\;\lambda\in\Omega_{\lambda_0},
\end{equation}
where
\begin{equation}
\label{3.9}
	\mathfrak{L}_{1}(\lambda) :=\left(\!\! \begin{array}{cc}
	I_{R[\mathfrak{L}(\lambda_0)]} & 0 \\[1ex]
	L_{21}(\lambda)L^{-1}_{11}(\lambda) & I_{R[\mathfrak{L}(\lambda_0)]^{\perp}}
	\end{array}\!\! \right), \;\; \mathfrak{L}_{2}(\lambda)  :=\left(\!\!\begin{array}{cc}
	I_{N[\mathfrak{L}(\lambda_0)]^{\perp}} & L^{-1}_{11}(\lambda)L_{12}(\lambda) \\[1ex]
	0 & I_{N[\mathfrak{L}(\lambda_0)]} 	\end{array}\!\!\right).
\end{equation}
Therefore, inverting \eqref{3.8} yields
\begin{equation}
\label{iii.10}
 {L}_{11}^{-1}(\lambda)\oplus[\mathscr{S}_{\mathfrak{L}(\lambda_{0}),(P,Q)}(\mathfrak{L}(\lambda))]^{-1}
 =\mathfrak{L}_{2}(\lambda)\mathfrak{L}^{-1}(\lambda)\mathfrak{L}_{1}(\lambda)
\end{equation}
for all $\lambda\in \Omega_{\lambda_0}\setminus\{\lambda_0\}$. Since
$$
  \mathfrak{L}(\Omega_{\lambda_{0}}\setminus\{\lambda_{0}\})\subset GL(U,V),
$$
by Theorem \ref{th3.2},  we have that, for every $\lambda\in\Omega_{\lambda_{0}}\setminus \{\lambda_{0}\}$,
$$
  \mathscr{S}_{\mathfrak{L}(\lambda_{0}),(P,Q)}(\mathfrak{L}(\lambda))\in GL(N[\mathfrak{L}(\lambda_0)],R[\mathfrak{L}(\lambda_0)]^{\perp}).
$$
On the one hand, since $\lambda_{0}\in\Alg_{\kappa}(\mathfrak{L})$, it is apparent that
\begin{equation}
\label{iii.11}
	\|\mathfrak{L}_{2}(\lambda)\mathfrak{L}^{-1}(\lambda)\mathfrak{L}_1(\lambda)\|\leq C\|\mathfrak{L}^{-1}(\lambda)\|\leq\frac{C}{|\lambda-\lambda_{0}|^{\kappa}}.
\end{equation}
Moreover,
\begin{equation}
\label{iii.12}
\|[\mathscr{S}_{\mathfrak{L}(\lambda_{0}),(P,Q)}(\mathfrak{L}(\lambda))]^{-1}\| \leq
 \|L_{11}^{-1}(\lambda)\oplus[\mathscr{S}_{\mathfrak{L}(\lambda_{0}),(P,Q)}(\mathfrak{L}(\lambda))]^{-1}\|.
\end{equation}
Consequently, taking norms in \eqref{iii.10}, we may conclude from \eqref{iii.12} and
\eqref{iii.11} that
\begin{equation*}
 \|[\mathscr{S}_{\mathfrak{L}(\lambda_{0}),(P,Q)}(\mathfrak{L}(\lambda))]^{-1}\|
 \leq\frac{C}{|\lambda-\lambda_{0}|^{k}}\quad \hbox{for all} \;\; \lambda\in\Omega_{\lambda_{0}}\setminus\{\lambda_{0}\},
\end{equation*}
which ends the proof of \eqref{3.7}.
\par
According to \eqref{3.8},  we already know that
\begin{equation}
\label{iii.13}
  \mathfrak{L}(\lambda)=\mathfrak{L}_{1}(\lambda)\left[ L_{11}(\lambda)\oplus
  \mathscr{S}_{\mathfrak{L}(\lambda_{0}),(P,Q)}(\mathfrak{L}(\lambda))\right]
  \mathfrak{L}_{2}(\lambda),
\end{equation}
with $\mathfrak{L}_{1}(\lambda), \mathfrak{L}_{2}(\lambda)\in GL(U,V)$ for all $\lambda\in\Omega_{\lambda_0}$. Thus, by the properties of $\chi$ listed at the end of
Section 3, it is apparent that
$$
   \chi[\mathfrak{L}_{1},\lambda_{0}]=\chi[\mathfrak{L}_{2},\lambda_{0}]=0.
$$
Moreover, since $L_{11}(\lambda)\in GL(U,V)$ for all $\lambda\in\Omega_{\lambda_{0}}$,
\begin{align*}
\chi[L_{11}(\lambda)\oplus(\mathscr{S}_{\mathfrak{L}(\lambda_{0}),(P,Q)}\circ\mathfrak{L}),\lambda_{0}]
 & =\chi[L_{11}(\lambda),\lambda_{0}]+\chi[\mathscr{S}_{\mathfrak{L}(\lambda_{0}),(P,Q)}
\circ\mathfrak{L},\lambda_{0}] \\ & =\chi[\mathscr{S}_{\mathfrak{L}(\lambda_{0}),(P,Q)}
\circ\mathfrak{L},\lambda_{0}].
\end{align*}
Observe that $\chi[\mathscr{S}_{\mathfrak{L}(\lambda_{0}),\mathcal{P}}\circ\mathfrak{L},\lambda_{0}]$ is well defined by \eqref{3.7}. Therefore,  applying the product formula to \eqref{iii.13} yields
\begin{equation*}
\chi[\mathfrak{L},\lambda_{0}]=\chi[\mathfrak{L}_{1},\lambda_{0}]+
\chi[L_{11}\oplus(\mathscr{S}_{\mathfrak{L}(\lambda_{0}),(P,Q)}\circ\mathfrak{L}),\lambda_{0}]
+\chi[\mathfrak{L}_{2},\lambda_{0}]=
\chi[\mathscr{S}_{\mathfrak{L}(\lambda_{0}),(P,Q)}\circ\mathfrak{L},\lambda_{0}].
\end{equation*}
Finally, since $\mathscr{S}_{\mathfrak{L}(\lambda_{0}),(P,Q)}\circ\mathfrak{L}$ is finite-dimensional and $L_{11}(\lambda)$ is an isomorphism, it follows from \eqref{2.4} and Definition \ref{de3.1}  that
\begin{align*}
\chi[\mathscr{S}_{\mathfrak{L}(\lambda_{0}),(P,Q)}\circ\mathfrak{L},\lambda_{0}]& =\ord_{\lambda_{0}}\mathcal{D}\mathscr{S}_{\mathfrak{L}(\lambda_{0}),(P,Q)}(\mathfrak{L}(\lambda))
\\ &= \ord_{\lambda_{0}}\mathcal{D} \mathscr{S}_{\mathfrak{L}(\lambda_{0}),(P,Q)}(\mathfrak{L}(\lambda))
\cdot\mathcal{D} L_{11}(\lambda) \\
& =\ord_{\lambda_{0}}\mathcal{D}_{\mathfrak{L}(\lambda_{0}),(P,Q)}(\mathfrak{L}(\lambda)).
\end{align*}
Lastly, taking into account that
$$
   \chi[\mathscr{S}_{\mathfrak{L}(\lambda_{0}),(P,Q)}\circ\mathfrak{L},\lambda_{0}] =
    \ord_{\lambda_{0}}\det\mathscr{S}_{\mathfrak{L}(\lambda_{0}),(P,Q)}(\mathfrak{L}(\lambda)),
$$
the proof is completed.
\end{proof}

To conclude this section, we are going to complete the proof of Theorem \ref{th1.1}. Note that Theorem \ref{th3.3} establishes the validity of the first identity of \eqref{1.7}. Assume  $\mathfrak{L}\in\mathcal{C}(\Omega,\Phi_{0}(U,V))$, let $\lambda_0\in\Omega\cap\Sigma(\mathfrak{L})$ be an isolated eigenvalue, and consider a pair $(P,Q)$ of $\mathfrak{L}(\lambda_0)$-projections.
According to \eqref{3.8}, we already know that
\begin{equation*}
	\mathfrak{L}(\lambda)=\mathfrak{L}_1(\lambda) \left(\begin{array}{cc}
	L_{11}(\lambda)& 0 \\[1ex] 	0 & \mathscr{S}_{\mathfrak{L}(\lambda_{0}),(P,Q)}(\mathfrak{L}(\lambda))
	\end{array}\right)\mathfrak{L}_2(\lambda),
\end{equation*}
where
\begin{equation*}
	\mathfrak{L}_{1}(\lambda) :=\left(\begin{array}{cc}
	I_{R[\mathfrak{L}(\lambda_0)]} & 0 \\[1ex]
	L_{21}(\lambda)L^{-1}_{11}(\lambda) & I_{R[\mathfrak{L}(\lambda_0)]^{\perp}}
	\end{array}\right), \;\; \mathfrak{L}_{2}(\lambda)  :=\left(\begin{array}{cc}
	I_{N[\mathfrak{L}(\lambda_0)]^{\perp}} & L^{-1}_{11}(\lambda)L_{12}(\lambda) \\[1ex]
	0 & I_{N[\mathfrak{L}(\lambda_0)]} 	\end{array}\right).
\end{equation*}
Thus, in a perforated neighborhood of $\lambda_0$ the inverse $\mathfrak{L}^{-1}(\lambda)$ is given through
\begin{equation*}
	\left(\!\!\!\begin{array}{cc}
	I_{N[\mathfrak{L}(\lambda_0)]^{\perp}} & -L_{11}^{-1}(\lambda)L_{12}(\lambda) \\[1ex]
	0 & I_{N[\mathfrak{L}(\lambda_0)]}  	\end{array}\!\!\! \right)\!\!\!\left(\!\!\!\begin{array}{cc}
	L_{11}^{-1}(\lambda)& 0 \\[1ex] 	
    0 & \mathscr{S}^{-1}_{\mathfrak{L}(\lambda_{0}),(P,Q)}(\mathfrak{L}(\lambda))
	\end{array}\!\!\!\right)\!\!\!\left(\!\!\!\begin{array}{cc} I_{R[\mathfrak{L}(\lambda_0)]}
     & 0 \\[1ex] 	-L_{21}(\lambda)L^{-1}_{11}(\lambda) & I_{R[\mathfrak{L}(\lambda_0)]^\perp}
	\end{array}\!\!\!\right).
\end{equation*}
Multiplying these operator matrices, it becomes apparent that
\begin{equation*}
	\mathfrak{L}^{-1}= \left(\begin{array}{cc}	 L_{11}^{-1}+L_{11}^{-1}L_{12}\mathscr{S}^{-1}_{\mathfrak{L}(\lambda_{0}),(P,Q)}(\mathfrak{L})  L_{21}L_{11}^{-1} & - L_{11}^{-1}L_{12} \mathscr{S}^{-1}_{\mathfrak{L}(\lambda_{0}),(P,Q)} (\mathfrak{L}) \\-\mathscr{S}^{-1}_{\mathfrak{L}(\lambda_{0}),(P,Q)}(\mathfrak{L}) L_{21}L_{11}^{-1} & \mathscr{S}^{-1}_{\mathfrak{L}(\lambda_{0}),(P,Q)}(\mathfrak{L}) \end{array}\right).
\end{equation*}
Therefore, since
\begin{equation*}
	\mathfrak{L}^{-1}(\lambda)=\left(\begin{array}{cc}
	Q\mathfrak{L}^{-1}(\lambda) (I_{U}-P) & Q\mathfrak{L}^{-1}(\lambda)P \\[1ex] 	
    (I_{V}-Q)\mathfrak{L}^{-1}(\lambda)(I_{U}-P) & (I_{V}-Q)\mathfrak{L}^{-1}(\lambda)P \end{array}\right),
\end{equation*}
identifying the last entries of these two operator matrices, \eqref{1.6} holds.
Therefore, by Theorem \ref{th3.3}, \eqref{1.7} follows readily if $\mathfrak{L}\in\mathcal{A}_{\lambda_{0}}(\Omega_{\lambda_{0}},\Phi_{0}(X,Y))$. This ends the proof
of Theorem \ref{th1.1}.

\section{Proof of Theorem \ref{th1.2}}

\noindent First, we will  recover $\chi$ from the perspective
of Nonlinear Analysis, by adopting the point of view of Bifurcation Theory. Actually, this was the methodology of Esquinas and L\'{o}pez-G\'{o}mez \cite{ELG} to construct $\chi$. In applications one is naturally
lead to deal with nonlinear equations of the form
\begin{equation}
\label{iv.1}
  \mathscr{F}(\lambda,u)\equiv \mathfrak{L}(\lambda)u+\mathfrak{N}(\lambda,u)=0,
\end{equation}
where $\mathfrak{L}\in \mathcal{C}^1(\Omega,\Phi_{0}(U,V))$ for some subdomain $\Omega$ of
$\mathbb{K}$, and $\mathfrak{N}\in \mathcal{C}^1(\Omega\times U,V)$ satisfies
\begin{equation}
\label{iv.2}
  \mathfrak{N}(\lambda,0)=0\quad\hbox{and}\quad D_u\mathfrak{N}(\lambda,0)=0
\end{equation}
for all $\lambda$ in a neighborhood of a given $\lambda_0\in\Omega$. Then,
$\mathscr{F}(\lambda,0)=0$ for $\lambda$ sufficiently close to $\lambda_0$ and, for any
given pair $(P,Q)$ of $\mathfrak{L}(\lambda_0)$-projections, the equation \eqref{iv.1}
can be equivalently written as
\begin{equation}
\label{iv.3}
   \left\{ \begin{array}{l}
    Q \mathfrak{L}(\lambda)(x+y)+Q \mathfrak{N}(\lambda,x+y)= 0,\\[1ex]
    (I_V-Q) \mathfrak{L}(\lambda)(x+y)+(I_V-Q) \mathfrak{N}(\lambda,x+y)= 0,
    \end{array}\right.
\end{equation}
where we are setting
$$
  x :=Pu\in N[\mathfrak{L}(\lambda_0)],\quad y:=(I_U-P)u \in R[I_U-P]\equiv N[\mathfrak{L}(\lambda_0)]^\perp.
$$
Subsequently, we consider the map $\mathscr{G}$ defined by
$$
  \mathscr{G}(\lambda,x,y) :=  Q \mathfrak{L}(\lambda)(x+y)+Q \mathfrak{N}(\lambda,x+y)
$$
for every $(\lambda,x,y)$ in a neighborhood, $\mathcal{O}$,  of $(\lambda_0,0,0)$ in
$$
   \Omega\times N[\mathfrak{L}(\lambda_0)]\times N[\mathfrak{L}(\lambda_0)]^\perp\thickapprox \Omega\times U.
$$
The map $\mathscr{G}$ is of class $\mathcal{C}^1(\mathcal{O})$ and it satisfies
$$
  D_y\mathscr{G}(\lambda,0,0)=Q\mathfrak{L}(\lambda)(I_U-P)\equiv L_{11}(\lambda)
$$
for $\lambda$ in a certain neighborhood, $\tilde\Omega\subset\Omega$,  of $\lambda_0$,
where we can assume that
$$
   L_{11}(\lambda)\in GL(N[\mathfrak{L}(\lambda_0)]^{\perp}, R[\mathfrak{L}(\lambda_0)]).
$$
Moreover, $\mathscr{G}(\lambda,0,0)=0$ for all $\lambda\in \tilde \Omega$. Thus, since
$$
 D_y\mathscr{G}(\lambda,0,0)=L_{11}(\lambda) \in GL(N[\mathfrak{L}(\lambda_0)]^{\perp},R[\mathfrak{L}(\lambda_0)]),
$$
by the Implicit Function Theorem, there exist a neighborhood,
$\mathcal{U}$, of $(\lambda_0,0)$ in $\mathbb{K}\times N[\mathfrak{L}(\lambda_0)]$ and a
map $\mathscr{Y}\in\mathcal{C}^1(\mathcal{U},R[\mathfrak{L}(\lambda_0)])$  such that
\begin{equation*}
  \mathscr{Y}(\lambda,0)=0\quad\hbox{and}\quad \mathscr{G}(\lambda,x,\mathscr{Y}(\lambda,x))=0
\end{equation*}
for all $(\lambda,x)\in \mathcal{U}$, and these are the unique zeroes of $\mathscr{G}$
in $\mathcal{U}$. So, $\mathscr{Y}(\lambda,0)=0$ for $\lambda$ sufficiently close to $\lambda_0$,
because $\mathscr{G}(\lambda,0,0)=0$. Furthermore, by the definition of $\mathscr{G}$,  we have that, for every $(\lambda,x)\in\mathcal{U}$,
\begin{equation}
\label{iv.4}
 Q \mathfrak{L}(\lambda)(x+\mathscr{Y}(\lambda,x))+Q \mathfrak{N}(\lambda,x+\mathscr{Y}(\lambda,x))=0.
\end{equation}
Thus, thanks to \eqref{iv.2}, differentiating \eqref{iv.4}  with respect to $x$ yields to
$$
   Q \mathfrak{L}(\lambda)(x +D_x \mathscr{Y}(\lambda,0)x)=0\quad \hbox{for all}\;\;
   x\in N[\mathfrak{L}(\lambda_0)].
$$
Equivalently,
$$
  Q\mathfrak{L}(\lambda)(I_U-P) D_x \mathscr{Y}(\lambda,0)x = -Q \mathfrak{L}(\lambda) P x
  \quad \hbox{for all}\;\;    x\in N[\mathfrak{L}(\lambda_0)].
$$
Consequently, for sufficiently small $\tilde \Omega$, we find that,
for every $\lambda\in \tilde \Omega$,
\begin{equation}
\label{iv.5}
  D_x \mathscr{Y}(\lambda,0) = -\left[  Q\mathfrak{L}(\lambda)(I_U-P)\right]^{-1} Q \mathfrak{L}(\lambda) P=-L_{11}^{-1}(\lambda) L_{12}(\lambda).
\end{equation}
Consequently, substituting $y$ by $\mathscr{Y}(\lambda,x)$  into the
second equation of \eqref{iv.3}, it turns out that solving \eqref{iv.1} locally at $(\lambda,u)=(\lambda_0,0)$ is equivalent to solve the  equation
\begin{equation}
\label{iv.6}
  (I_V-Q) \mathfrak{L}(\lambda)(x+\mathscr{Y}(\lambda,x))+(I_V-Q) \mathfrak{N}(\lambda,x+\mathscr{Y}(\lambda,x))= 0
\end{equation}
in a neighborhood of $(\lambda_0,0)$ in $\mathbb{K}\times N[\mathfrak{L}(\lambda_0)]$. Equation
\eqref{iv.6} is often referred to as the \emph{bifurcation equation} of \eqref{iv.1} at $(\lambda_0,0)$.
According to \cite[Le. 3.1.1]{LG01}, $(\lambda,u)=(\lambda_0,0)$ is a bifurcation point of
\eqref{iv.1} from $(\lambda,u)=(\lambda,0)$ if and only if $(\lambda,x)=(\lambda_0,0)$ is a bifurcation point of \eqref{iv.5} from $(\lambda,x)=(\lambda,0)$. As the linearization of the left hand side
of \eqref{iv.6} at $(\lambda,x)=(\lambda,0)$ is given by the linear operator
\begin{equation*}
  \mathscr{B}(\lambda)x:=(I_V-Q) \mathfrak{L}(\lambda)\left( x+ D_x \mathscr{Y}(\lambda,0)x\right)\quad
  \hbox{for all}\;\; x \in N[\mathfrak{L}(\lambda)],
\end{equation*}
and, due to \eqref{iv.5}, $\mathscr{B}(\lambda)$ can be expressed as
\begin{equation*}
  \mathscr{B}(\lambda):=(I_V-Q) \mathfrak{L}(\lambda)\left[ P - (I_U-P) L_{11}^{-1}(\lambda)L_{12}(\lambda)\right],
\end{equation*}
by the definitions of the $L_{ij}(\lambda)$'s, it becomes apparent that
\begin{equation}
\label{iv.7}
  \mathscr{B}(\lambda):=L_{22}(\lambda) - L_{21}(\lambda)
  L_{11}^{-1}(\lambda)L_{12}(\lambda)=
  \mathscr{S}_{\mathfrak{L}(\lambda),(P,Q)}(\mathfrak{L}(\lambda)).
\end{equation}
Therefore, the first part of Theorem \ref{th1.2} and, in particular, \eqref{1.8} is proven. Naturally,
\eqref{1.9} follows readily by combining \eqref{1.8} with Theorem \ref{th1.1}. The equivalence between (b), (c), (d), and (e) is a direct consequence from \eqref{1.7} and \eqref{1.9}, and the equivalence of (a) and (b) is \cite[Th. 4.3.4]{LG01}.
\par
Finally, suppose that $\chi[\mathfrak{L},\lambda_0]$ is an odd integer. Then, due to \eqref{1.9},
$\det \mathscr{B}(\lambda)$ changes sign as $\lambda$ crosses $\lambda_0$. Hence, also the Brouwer degree of $(\lambda,0)$ changes as $\lambda$ crosses $\lambda_0$. Thus, thanks to  \cite[Th. 6.2.1]{LG01}, the bifurcation equation \eqref{iv.6} possesses a (local) continuum of non-trivial solutions emanating from $(\lambda,x)=(\lambda,0)$ at $\lambda_0$. Naturally, by the local structure of the solutions
of \eqref{iv.1} in a neighborhood of $(\lambda,u)=(\lambda_0,0)$, it becomes apparent that
\eqref{iv.1} also has a (local) continuum of non-trivial solutions emanating from $(\lambda,u)=(\lambda,0)$ at $\lambda=\lambda_0$. Consequently,  Zorn's lemma shows the existence of the component $\mathscr{C}$. This ends the proof of Theorem \ref{th1.2}.

\section{The classical multiplicity from a geometrical perspective}

\noindent The main goal of this section is establishing a, rather hidden, bi-association between
the classical algebraic multiplicity and the concept of intersection index of varieties in
Algebraic Geometry. Naturally, this might open new challenging perspectives in both fields,
besides incrementing their number of potential applications.
The next, and final, section will extend the findings of this one to cover the generalized algebraic multiplicity $\chi$.  As for algebraic purposes one needs to work in algebraically closed fields,  throughout the rest of this paper we will assume that  $\mathbb{K}=\mathbb{C}$.
\par
Subsequently, we identify $\mathcal{L}(\mathbb{C}^{N})\simeq\mathbb{C}^{N^{2}}$ and regard the determinant map,
$$
   \mathcal{D}: \mathcal{L}(\mathbb{C}^{N})\to\mathbb{C}, \qquad \mathcal{D}(T):=\det T \;\;
   \hbox{for every}\;\;  T\in \mathcal{L}(\mathbb{C}^{N}),
$$
as an homogeneous  polynomial map of degree $N$. Using this identification, every $T\in \mathcal{L}(\mathbb{C}^{N})$ can be expressed as en element of
$\mathbb{C}^{N^2}$ through the identification
\begin{equation*}
T = \left(\begin{array}{cccc}
x_{1} & x_{2} & \cdots & x_{N} \\
x_{N+1} & x_{N+2} & \cdots & x_{2N} \\
\vdots & \vdots & \ddots & \vdots \\
x_{N^{2}-N+1} & x_{N^{2}-N+2} & \cdots & x_{N^{2}}
\end{array}\right)=\left(x_1,x_2,\ldots,x_{N^2-1},x_{N^2}\right).
\end{equation*}
Then, for every $1\leq k\leq N^{2}$, we define the $k$\textit{-th differential functional} of the determinant map at a given operator $L\in\mathcal{L}(\mathbb{C}^{N})$ as the map
$\mathcal{D}^{k}(L):  \mathcal{L}(\mathbb{C}^{N})  \to  \mathbb{C}$ defined by
\begin{equation*}
\mathcal{D}^{k}(L)[T]:=\sum_{1\leq j_{1},j_{2},\cdots,j_{k}\leq N^{2}}\left( \frac{\partial^{k}\mathcal{D}}{\partial x_{j_{1}}\partial x_{j_{2}}\cdots\partial x_{j_{k}}}\right) (L) \cdot  x_{j_{1}}x_{j_{2}}\cdots x_{j_{k}}
\end{equation*}
for all $T\in \mathcal{L}(\mathbb{C}^{N})$.  Note that, for any given $L\in\mathcal{L}(\mathbb{C}^{N})$, the Taylor expansion of the determinant map, $\mathcal{D}$, centered at $L$, is given through
$$
  \mathcal{D}(T)=\mathcal{D}(L)+\sum_{k=1}^{N^2} \frac{1}{k!}\mathcal{D}^{k}(L)[T-L]\qquad \hbox{for all}\;\; T\in  \mathcal{L}(\mathbb{C}^{N}).
$$
In particular, $\mathcal{D}^{k}(L)$ can be viewed as an homogeneous polynomial of degree $k$. The associated algebraic variety
$$
  V(\mathcal{D}^{k}(L))= (\mathcal{D}^{k}(L))^{-1}(0)
$$
will be subsequently refereed to as the $k$\emph{-th tangent variety} of the determinant map at $L\in\mathcal{L}(\mathbb{C}^{N})$, and will be denoted by $\mathscr{T}_{L}^{k}\mathcal{D}$. As $\mathscr{T}^{1}_{L}\mathcal{D}$ is the tangent space to  $D^{-1}(0)$ at $L$, the variety $\mathscr{T}^{k}_{L}\mathcal{D}$ can be viewed as a generalization of the classical tangent space at a given point.   The following technical result can be easily obtained from the chain rule.

\begin{lemma}
\label{le5.1}
Let $T\in\mathcal{L}(\mathbb{C}^{N})$ and $\mathfrak{L}(\lambda)=\lambda I_{N}-T$ for all $\lambda\in \mathbb{C}$. Then,  for every $r\in\mathbb{N}$, the following identity holds
\begin{equation}
\label{5.1}
	\frac{d^r}{d\lambda^r}( \mathcal{D}\mathfrak{L}(\lambda))\equiv (\mathcal{D}\mathfrak{L})^{(r)}(\lambda)= \mathcal{D}^{r}(\mathfrak{L}(\lambda))[I_{N}].
\end{equation}
\end{lemma}
\begin{proof}
For notational convenience, we will identify $x_{i,j}$ with $x_{(i-1)N+j}$ for any $N\times N$ matrix $(x_{i,j})_{1\leq i,j\leq N}$. By definition,
\begin{equation*}
	\mathcal{D}^{r}(\mathfrak{L}(\lambda))[I_{N}] =\sum_{1\leq i_{1},i_{2},\cdots,i_{r}\leq N}\left(\frac{\partial^{r}\mathcal{D}}{\partial x_{i_{1},i_{1}} \partial x_{i_{2},i_{2}} \cdots \partial x_{i_{r},i_{r}}}\right)(\mathfrak{L}(\lambda)).
\end{equation*}
Thus, \eqref{5.1} can be equivalently written in the form
\begin{equation}
\label{5.2}
	(\mathcal{D}\mathfrak{L})^{(r)}(\lambda)=\sum_{1\leq i_{1},i_{2},\cdots,i_{r}\leq N}\left(\frac{\partial^{r}\mathcal{D}}{\partial x_{i_{1},i_{1}} \partial x_{i_{2},i_{2}} \cdots \partial x_{i_{r},i_{r}}}\right)(\mathfrak{L}(\lambda)).
\end{equation}
The proof of \eqref{5.2} will proceed by induction on $r$. Suppose $r=1$. Then, since $\mathfrak{L}'(\lambda)=I_{N}$, we have that
\begin{equation*}	 (\mathcal{D}\mathfrak{L})^{'}(\lambda)=\nabla\mathcal{D}(\mathfrak{L}(\lambda))\circ\mathfrak{L}'(\lambda)
=\sum_{i=1}^{N^{2}}\frac{\partial\mathcal{D}}{\partial x_{i}}(\mathfrak{L}(\lambda))\cdot [\mathfrak{L}'(\lambda)]_{i}=\sum_{i_{1}=1}^{N}\frac{\partial\mathcal{D}}{\partial x_{i_{1},i_{1}}}(\mathfrak{L}(\lambda)),
\end{equation*}
where $[\mathfrak{L}'(\lambda)]_{i}$ is the $i$-th term of the matrix $\mathfrak{L}'(\lambda)$. Thus,
\eqref{5.2} holds for $r=1$. Now, suppose that \eqref{5.2} is satisfied for some $r=n-1$ with $n\geq 2$, i.e.,
\begin{equation}
\label{5.3}
	(\mathcal{D}\mathfrak{L})^{(n-1)}(\lambda)=\sum_{1\leq i_{1},i_{2},\cdots,i_{n-1}\leq N}\left(\frac{\partial^{n-1}\mathcal{D}}{\partial x_{i_{1},i_{1}} \partial x_{i_{2},i_{2}} \cdots \partial x_{i_{n-1},i_{n-1}}}\right)(\mathfrak{L}(\lambda)).
\end{equation}
Then, differentiating with respect to $\lambda$ yields
\begin{align*}
	(\mathcal{D}\mathfrak{L})^{(n)}(\lambda)&=\frac{d}{d\lambda} \sum_{1\leq i_{1},i_{2},\cdots,i_{n-1}\leq N}\left(\frac{\partial^{n-1}\mathcal{D}}{\partial x_{i_{1},i_{1}} \partial x_{i_{2},i_{2}} \cdots \partial x_{i_{n-1},i_{n-1}}}\right)(\mathfrak{L}(\lambda)) \\
	&=\sum_{1\leq i_{1},i_{2},\cdots,i_{n-1}\leq N}\frac{d}{d\lambda} \left[\left(\frac{\partial^{n-1}\mathcal{D}}{\partial x_{i_{1},i_{1}} \partial x_{i_{2},i_{2}} \cdots \partial x_{i_{n-1},i_{n-1}}}\right)(\mathfrak{L}(\lambda))\right] \\
	&=\sum_{1\leq i_{1},i_{2},\cdots,i_{n-1}\leq N}\sum_{j=1}^{N^{2}}\left(\frac{\partial^{n}\mathcal{D}}{\partial x_{i_{1},i_{1}} \partial x_{i_{2},i_{2}} \cdots \partial x_{i_{n-1},i_{n-1}}\partial x_{j}}\right)(\mathfrak{L}(\lambda))\cdot [\mathfrak{L}'(\lambda)]_{j}\\
	&= \sum_{1\leq i_{1},i_{2},\cdots,i_{n}\leq N}\left(\frac{\partial^{n}\mathcal{D}}{\partial x_{i_{1},i_{1}} \partial x_{i_{2},i_{2}} \cdots \partial x_{i_{n},i_{n}}}\right)(\mathfrak{L}(\lambda)).
\end{align*}
This concludes the proof.
\end{proof}

The following result, illustrated by Figure \ref{F1}, gives a geometrical meaning to the classical algebraic multiplicity of an eigenvalue of a linear operator.  As  illustrated by Figure \ref{F1},
the algebraic multiplicity provides us with the minimal integer $\mathfrak{m}$  for which the straight line $\lambda I_N$ is not contained within the $\mathfrak{m}$-th tangent variety of the determinant map at $\mathfrak{L}(\lambda_{0})$, which has been denoted by $\mathscr{T}^{\mathfrak{m}}_{\mathfrak{L}(\lambda_{0})}\mathcal{D}$. Actually, Figure \ref{F1} illustrates an admissible case when $\lambda I_N$ lies entirely within the tangent variety  $\mathscr{T}^{1}_{\mathfrak{L}(\lambda_{0})}\mathcal{D}$ and hence,
$$
  \mathfrak{m}=\chi[\lambda I_N-T,\lambda_0] \geq 2.
$$

\begin{theorem}
\label{th5.2}
Let $T\in\mathcal{L}(\mathbb{C}^{N})$, pick $\lambda_{0}\in\sigma(T)$, and consider the straight line
\begin{equation*}
	\mathfrak{L}(\lambda)=\lambda I_{N}-T, \quad \lambda\in \mathbb{C}.
\end{equation*}
Then, $\mathfrak{m} := \chi[\mathfrak{L},\lambda_{0}]$ is the minimal integer $\mathfrak{m}\geq 1$ for which  the line $\lambda I_{N}$ is contained in $\mathscr{T}^{i}_{\mathfrak{L}(\lambda_{0})}\mathcal{D}$ for every $1\leq i < \mathfrak{m}$, but not for $i=\mathfrak{m}$.
\end{theorem}

\begin{proof}
By \eqref{5.3},  $\mathfrak{m}=\chi[\mathfrak{L},\lambda_{0}]$ if and only if
$$
   (\mathcal{D}\mathfrak{L})^{(i)}(\lambda_{0})=0 \;\;  \text{ for each } \;\; i\in\{1,...,\mathfrak{m}-1\} \;\; \hbox{and}\;\; (\mathcal{D}\mathfrak{L})^{(\mathfrak{m})}(\lambda_{0})\neq 0.
$$
Therefore, owing to \eqref{5.1}, it becomes apparent that
$$
  \mathcal{D}^{i}(\mathfrak{L}(\lambda_{0}))[I_{N}]=0\;\;  \text{ for each }\;\;  i\in\{1,...,\mathfrak{m}-1\}\;\; \hbox{and}\;\; \mathcal{D}^{\mathfrak{m}}(\mathfrak{L}(\lambda_{0}))[I_{N}]\neq 0.
$$
We conclude the proof by noticing that the line $\lambda I_{N}$ is contained in the variety
$$
  \mathscr{T}^{i}_{\mathfrak{L}(\lambda_{0})}\mathcal{D}=\left(\mathcal{D}^{i}(\mathfrak{L}(\lambda_{0}))
  \right)^{-1}(0)
$$
if and only if
\begin{equation*}
\mathcal{D}^{i}(\mathfrak{L}(\lambda_{0}))[\lambda I_{N}]=0 \text{ for all } \lambda\in\mathbb{C}.
\end{equation*}
As the map $\mathcal{D}^{i}(\mathfrak{L}(\lambda_{0}))$ is linear, this occurs if and only if
\begin{equation*}
  \mathcal{D}^{i}(\mathfrak{L}(\lambda_{0}))[I_{N}]=0.
\end{equation*}
The proof is complete.
\end{proof}

\begin{center}

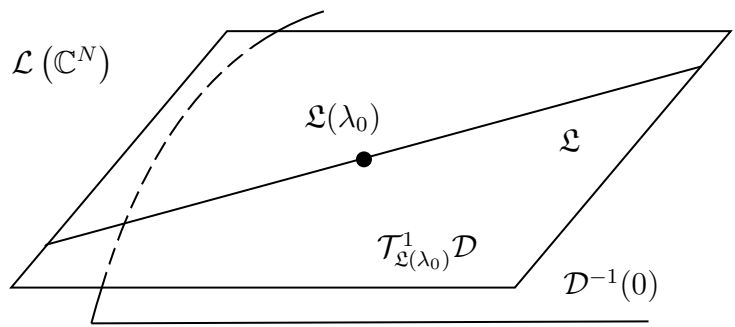
\begin{figure}[ht]

\tikzset{every picture/.style={line width=0.75pt}} 

\begin{tikzpicture}[x=0.75pt,y=0.75pt,yscale=-1,xscale=1]

\draw    (145.5,185) .. controls (158.5,137) and (186.5,52) .. (261.5,28) ;

\draw    (145.5,185) -- (423.5,184) ;

\draw   (213.2,38) -- (464.5,38) -- (356.8,167) -- (105.5,167) -- cycle ;
\draw  [fill={rgb, 255:red, 0; green, 0; blue, 0 }  ,fill opacity=1 ] (278,102.5) .. controls (278,100.57) and (279.57,99) .. (281.5,99) .. controls (283.43,99) and (285,100.57) .. (285,102.5) .. controls (285,104.43) and (283.43,106) .. (281.5,106) .. controls (279.57,106) and (278,104.43) .. (278,102.5) -- cycle ;
\draw    (121.9,145.73) -- (449.9,55.73) ;

\draw  [color={rgb, 255:red, 255; green, 255; blue, 255 }  ,draw opacity=1 ][line width=3] [line join = round][line cap = round] (217.24,46.47) .. controls (219.98,46.47) and (228.24,50.51) .. (228.24,52.47) ;
\draw  [color={rgb, 255:red, 255; green, 255; blue, 255 }  ,draw opacity=1 ][line width=3] [line join = round][line cap = round] (207.24,55.47) .. controls (212.91,55.47) and (214.37,61.47) .. (217.24,61.47) ;
\draw  [color={rgb, 255:red, 255; green, 255; blue, 255 }  ,draw opacity=1 ][line width=3] [line join = round][line cap = round] (198.24,65.47) .. controls (200,65.47) and (206.24,70.98) .. (206.24,72.47) ;
\draw  [color={rgb, 255:red, 255; green, 255; blue, 255 }  ,draw opacity=1 ][line width=3] [line join = round][line cap = round] (187.24,76.47) .. controls (189.36,76.47) and (197.24,82.04) .. (197.24,83.47) ;
\draw  [color={rgb, 255:red, 255; green, 255; blue, 255 }  ,draw opacity=1 ][line width=3] [line join = round][line cap = round] (176.24,88.47) .. controls (178.89,88.47) and (187.24,93.95) .. (187.24,96.47) ;
\draw  [color={rgb, 255:red, 255; green, 255; blue, 255 }  ,draw opacity=1 ][line width=3] [line join = round][line cap = round] (170.24,101.47) .. controls (171.9,101.47) and (182.24,107.81) .. (182.24,108.47) ;
\draw  [color={rgb, 255:red, 255; green, 255; blue, 255 }  ,draw opacity=1 ][line width=3] [line join = round][line cap = round] (163.24,113.47) .. controls (165.27,115.5) and (177.24,119.09) .. (177.24,121.47) ;
\draw  [color={rgb, 255:red, 255; green, 255; blue, 255 }  ,draw opacity=1 ][line width=3] [line join = round][line cap = round] (157.24,125.47) .. controls (161.63,125.47) and (163.58,129.47) .. (166.24,129.47) ;
\draw  [color={rgb, 255:red, 255; green, 255; blue, 255 }  ,draw opacity=1 ][line width=3] [line join = round][line cap = round] (153.24,143.47) .. controls (157.65,143.47) and (159.29,147.47) .. (163.24,147.47) ;
\draw  [color={rgb, 255:red, 255; green, 255; blue, 255 }  ,draw opacity=1 ][line width=3] [line join = round][line cap = round] (148.24,153.47) .. controls (150.53,155.76) and (160.24,160.49) .. (160.24,162.47) ;
\draw  [color={rgb, 255:red, 255; green, 255; blue, 255 }  ,draw opacity=1 ][line width=3] [line join = round][line cap = round] (120.83,143.33) .. controls (120.5,143.33) and (120.17,143.33) .. (119.83,143.33) ;
\draw  [color={rgb, 255:red, 255; green, 255; blue, 255 }  ,draw opacity=1 ][line width=3] [line join = round][line cap = round] (120.83,144.33) .. controls (120.5,144.33) and (120.17,144.33) .. (119.83,144.33) ;

\draw (405,166) node   {$\mathcal{D}^{-1}(0)$};
\draw (271,82) node   {$\mathfrak{L}( \lambda _{0})$};
\draw (313,149) node   {$\mathcal{T}^{1}_{\mathfrak{L}( \lambda _{0})}\mathcal{D}$};
\draw (384,93) node   {$\mathfrak{L}$};
\draw (131,55) node   {$\mathcal{L}\left(\mathbb{C}^{N}\right)$};

\end{tikzpicture}

\caption{Geometrical Meaning of Classical Algebraic Multiplicity}
\label{F1}
\end{figure}
\end{center}

Theorem \ref{th5.2} leads us to think of the existence of a sharp connection between Spectral Theory and Algebraic Geometry, which had been already stated in Theorem \ref{th1.3}. The rest of this section is devoted to deliver a proof of this theorem.
\par
Let $T\in\mathcal{L}(\mathbb{C}^{N})$, pick $\lambda_{0}\in\sigma(T)$, and consider the straight line
\begin{equation*}
	\mathfrak{L}(\lambda)=\lambda I_{N}-T, \quad \lambda\in \mathbb{C}.
\end{equation*}
As usual in this section, we express the operator $T\in\mathcal{L}(\mathbb{C}^{N})$ in matrix form
\begin{equation*}
	T\sim\left(\begin{array}{cccc}
	t_{1} & t_{2} & \cdots & t_{N} \\
	t_{N+1} & t_{N+2} & \cdots & t_{2N} \\
	\vdots & \vdots & \ddots & \vdots \\
	t_{N^{2}-N+1} & t_{N^{2}-N+2} & \cdots & t_{N^{2}}
	\end{array}\right),
\end{equation*}
i.e., setting $T=\left(t_{i,j}\right)_{1\leq i, j\leq N}$, we are identifying $t_{i,j}$ with $t_{(i-1)N+j}$.
\par
Naturally, the parametric equations of the line $\mathfrak{L}(\mathbb{C})$ in $\mathbb{C}^{N^{2}}$ are the following ones:
\begin{equation*}
	\left\{
	\begin{array}{ll}
	x_{i}=\lambda-t_{i} & \text{ for } i\in\{(m-1)N+m: m\in\{1,2,...,N\}\},\\
	x_{i}=-t_{i} & \text{ for } i\in\{1,2,...,N^{2}\}\backslash\{(m-1)N+m: m\in\{1,2,...,N\}\}.
	\end{array}\right.
\end{equation*}
Thus, by a simple computation involving the following identity
\begin{equation}
\label{5.4}
  \lambda=x_{1}+t_{1},
\end{equation}
it is easily seen that $ \mathfrak{L}(\mathbb{C})$ is the algebraic variety generated by the polynomials
\begin{equation*}
	\left\{
	\begin{array}{ll}
	\mathfrak{P}_{i-1}\equiv x_{i}-x_{1}+t_{i}-t_{1} & \text{ for } i\in\{(m-1)N+m: m\in\{2,...,N\}\},\\
	\mathfrak{P}_{i-1}\equiv x_{i}+t_{i} & \text{ for } i\in\{1,2,...,N^{2}\}\backslash\{(m-1)N+m: m\in\{1,2,...,N\}\}. \end{array} 	\right.
\end{equation*}
In other words,
$$
  \mathfrak{L}(\mathbb{C}) = V\left(\mathfrak{P}_{1},\ldots,\mathfrak{P}_{N^2-1}\right),
$$
i.e., the implicit equations of  $ \mathfrak{L}(\mathbb{C})$ are
\begin{equation*}
	\left\{
	\begin{array}{ll}
	x_{i}-x_{1}+t_{i}-t_{1}=0 & \text{ for } i\in\{(m-1)N+m: m\in\{2,...,N\}\},\\
	x_{i}+t_{i} =0 & \text{ for } i\in\{1,2,...,N^{2}\}\backslash\{(m-1)N+m: m\in\{1,2,...,N\}\}. \end{array} 	 \right.
\end{equation*}
According to Theorem 18.18 of Eisenbud \cite{E0}, the determinantal variety $\mathcal{D}^{-1}(0)$ is Cohen-Macaulay, as well as $\mathfrak{L}(\mathbb{C})$, because it is smooth
(see \cite[Sect. 18.2]{E0}). Therefore, by Eisenbud and Harris \cite[p. 48]{E},
\begin{equation*}
i(\mathcal{D}^{-1}(0),\mathfrak{L}(\mathbb{C});\mathfrak{L}(\lambda_{0}))=
\ell_{\mathcal{O}}(\mathcal{O}/(\mathfrak{D},\mathfrak{P}))
\end{equation*}
(see \eqref{1.11}), where $\mathcal{O}\equiv \mathcal{O}_{\mathfrak{L}(\lambda_{0}),\mathbb{C}^{N^{2}}}$ stands for  the local ring of $\mathbb{C}^{N^{2}}$ at $\mathfrak{L}(\lambda_0)$, $\mathfrak{D}$ is the ideal generated by the determinant map,
$$
  \mathfrak{D}:=(\mathcal{D}(x_{1},x_{2},\cdots,x_{N^{2}})),
$$
$\mathfrak{P}$ stands for the ideal generated by the polynomials $\mathfrak{P}_{j}$, $1\leq j\leq N^2-1$,
and $\ell_\mathcal{O}$ is the length of the module $\mathcal{O}/(\mathfrak{D},\mathfrak{P})$ over the ring $\mathcal{O}$. By \eqref{5.4},
\begin{align*}
(\mathfrak{D},\mathfrak{P}) & = (\mathcal{D}(x_{1},x_{2},\cdots,x_{N^{2}}),\mathfrak{P}_{1},\mathfrak{P}_{2},\cdots,\mathfrak{P}_{N^{2}-1})
\\ & =(\mathcal{D}[(x_{1}+t_{1})I_N-T],\mathfrak{P}_{1},\mathfrak{P}_{2},\cdots,\mathfrak{P}_{N^{2}-1})\\ &=(\mathcal{D}\mathfrak{L}(x_{1}+t_{1}),\mathfrak{P}_{1},\mathfrak{P}_{2},\cdots,\mathfrak{P}_{N^{2}-1}).
\end{align*}
On the other hand, by the definition of the classical algebraic geometry, $\mathfrak{m}=\chi[\lambda I_N-T,\lambda_{0}]$,  the following identity holds
\begin{equation*}
	\mathcal{D}\mathfrak{L}(x_{1}+t_{1})=(x_{1}+t_{1}-\lambda_{0})^{\mathfrak{m}}f(x_{1}+t_{1})
\end{equation*}
for some polynomial $f$ such that $f(\lambda_{0})\neq 0$. Thus, since  $f(x_{1}+t_{1})$ is a unit in the local ring $\mathcal{O}_{\mathfrak{L}(\lambda_{0}),\mathbb{C}^{N^{2}}}$, we find that
\begin{equation*}
(\mathfrak{D},\mathfrak{P}) = ((x_{1}+t_{1}-\lambda_{0})^{\mathfrak{m}},\mathfrak{P}_{1},\mathfrak{P}_{2},\cdots,\mathfrak{P}_{N^{2}-1}).
\end{equation*}
On the other hand,   by a routine calculation,
\begin{equation*}
   V((x_{1}+t_{1}-\lambda_{0})^{\mathfrak{m}},\mathfrak{P}_{1},
   \mathfrak{P}_{2},\cdots,\mathfrak{P}_{N^{2}-1})=\left\{ \mathfrak{L}(\lambda_{0})\right\}.
\end{equation*}
Therefore, by, e.g.,  Proposition 6 of Fulton \cite{F1}, we get the isomorphism
\begin{equation*}
	\frac{\mathcal{O}_{\mathfrak{L}(\lambda_{0}), \mathbb{C}^{N^{2}}}}{(\mathfrak{D},\mathfrak{P})}\simeq \frac{\mathbb{C}[x_{1},x_{2},...,x_{N^{2}}]}{(\mathfrak{D},\mathfrak{P})}.
\end{equation*}
Moreover, since the residue field of $\mathcal{O}$ is $\mathbb{C}$ and $\mathfrak{m}\cdot\mathcal{O}/(\mathfrak{D},\mathfrak{P})=0$, where $\mathfrak{m}$ is the maximal ideal of $\mathcal{O}$, necessarily, by  Proposition 5.9 of Bump \cite{Bu},
\begin{equation*}
\ell_{\mathcal{O}}(\mathcal{O}/(\mathfrak{D},\mathfrak{P}))=\dim_{\mathbb{C}}(\mathcal{O}/(\mathfrak{D},\mathfrak{P})).
\end{equation*}
Since in $\mathbb{C}[x_{1},x_{2},...,x_{N^{2}}]/(\mathfrak{D},\mathfrak{P})$ the following
$N^2$ identities hold true
\begin{equation*}
	\left\{\begin{array}{ll}
	(x_{1}+t_{1}-\lambda_0)^{\mathfrak{m}}=0 & \\
	 x_{i}=x_{1}+t_{1}-t_{i} & \text{ for } i\in\{(m-1)n+m: m\in\{2,...,n\}\},\\
	 x_{i}=-t_{i} & \text{ for } i\in\{1,2,...,n^{2}\}\backslash\{(m-1)n+m: m\in\{1,2,...,n\}\},
	\end{array}\right.
\end{equation*}
it is apparent, from the Newton binomial,  that $x_1^\mathfrak{m}$ depends linearly on $x_1^j$, $0\leq j\leq \mathfrak{m}-1$, as well as the $x_i$'s for all $i\in \{2,...,N^2\}$. Therefore, 	 	
\begin{equation*}
	\frac{\mathbb{C}[x_{1},x_{2},...,x_{N^{2}}]}{(\mathfrak{D},\mathfrak{P})}=\langle 1, x_{1}, x^{2}_{1}, \cdots, x^{\mathfrak{m}-1}_{1} \rangle_{\mathbb{C}}
\end{equation*}
and consequently, 	
\begin{align*}
  i(\mathcal{D}^{-1}(0),\mathfrak{L}(\mathbb{R});\mathfrak{L}(\lambda_{0}))& =\ell_{\mathcal{O}}(\mathcal{O}/(\mathfrak{D},\mathfrak{P}))=\dim_{\mathbb{C}}(\mathcal{O}/(\mathfrak{D},\mathfrak{P}))\\
  &=\dim_{\mathbb{C}}
  \left(\mathbb{C}[x_{1},x_{2},...,x_{N^{2}}]/(\mathfrak{D},\mathfrak{P})\right)=\mathfrak{m}.
\end{align*}
This ends the proof of Theorem \ref{th1.3}.

\section{Proof of Theorem \ref{th1.4}}

\noindent Thanks to  the existence of the local Smith  form for general analytic curves, $\mathfrak{L}$,
at  isolated singular values, $\lambda_0$, the algebraic multiplicity of these curves, $\chi[\mathfrak{L},\lambda_0]$, can be also expressed as
the intersection index of two algebraic varieties, as already stated by Theorem \ref{th1.4}.
This section consists of the proof of that theorem.
\par
Let $\mathfrak{L}\in \mathcal{A}_{\lambda_{0}}(\Omega_{\lambda_{0}},\Phi_{0}(X,Y))$ be, and fix a pair $(P,Q)$ of $\mathfrak{L}(\lambda_{0})$-projections. Then, thanks to Theorem \ref{th3.3}, we have that
\begin{align*}
	\chi[\mathfrak{L},\lambda_{0}]& =\ord_{\lambda_{0}}\mathcal{D}_{\mathfrak{L}(\lambda_{0}),(P,Q)}(\mathfrak{L}(\lambda))=
\ord_{\lambda_{0}}\mathcal{D}\mathscr{S}_{\mathfrak{L}(\lambda_{0}),(P,Q)}
(\mathfrak{L}(\lambda))\\ 	 &=\ord_{\lambda_{0}}\mathcal{D}[L_{22}(\lambda)-L_{21}(\lambda)L_{11}^{-1}(\lambda) L_{12}(\lambda)].
\end{align*}
For the sake of notation, let us identify
$$
   \mathcal{L}(N[\mathfrak{L}(\lambda_0)],R[\mathfrak{L}(\lambda_0)]^{\perp})
   \simeq\mathcal{L}(\mathbb{C}^{N})\quad
   \hbox{with}\quad N:=\dim N[\mathfrak{L}(\lambda_{0})].
$$
According to \eqref{3.7}, we already know that
$$
 \lambda_{0}\in\Alg_{\kappa}(\mathscr{S}_{\mathfrak{L}(\lambda_{0}),(P,Q)}\circ\mathfrak{L})
$$
for some $1\leq \kappa\leq r$, where $\mathcal{C}^r$ is the class of regularity
of $\mathfrak{L}(\lambda)$. Thus,  by Theorem 7.8.1 of L\'{o}pez-G\'{o}mez and Mora-Corral \cite{LGMC}, $\mathscr{S}_{\mathfrak{L}(\lambda_{0}),\mathcal{P}}\circ\mathfrak{L}$ admits a local Smith form at $\lambda_0$, i.e.,  there exist a neighborhood, $\Omega_{\lambda_{0}}$,  of $\lambda_0$ such that $\Omega_{\lambda_{0}}\subset \Omega$, two invertible curves $\mathfrak{E}, \mathfrak{F} \in\mathcal{C}(\Omega_{\lambda_{0}},GL(\mathbb{C}^{N}))$, and $N$ positive integers
$$
  \kappa_1\geq \cdots \geq \kappa_N \geq 1
$$
such that
\begin{equation}
\label{6.1}
\mathscr{S}_{\mathfrak{L}(\lambda_{0}),(P,Q)}(\mathfrak{L}(\lambda))=\mathfrak{E}(\lambda)\cdot \mathfrak{P}(\lambda) \cdot\mathfrak{F}(\lambda) \quad\hbox{for all}\;\;  \lambda\in\Omega_{\lambda_{0}},
\end{equation}
where $\mathfrak{P}(\lambda)$ is the matrix of order $n$ defined by
\begin{equation}
\label{6.2}
\mathfrak{P}(\lambda):= \mathrm{diag\,}\left( (\lambda-\lambda_0)^{\kappa_1},\ldots,
(\lambda-\lambda_0)^{\kappa_N}\right),
\end{equation}
which is the \emph{local Smith form} of the \emph{Schur operator} of $\mathfrak{L}(\lambda)$ at $\lambda_0$.
\par
On the other hand, according to  Lemma 10.1.1 of L\'{o}pez-G\'{o}mez and Mora-Corral \cite{LGMC}, there exist an integer $M\geq N$,  two invertible matrix polynomial curves $\mathfrak{P}_{1}, \mathfrak{P}_{2}\in\mathcal{H}(\Omega_{\lambda_{0}},GL(\mathbb{C}^{M}))$, and an operator $\mathscr{L}\in\mathcal{L}(\mathbb{C}^{M})$ such that
\begin{equation}
\label{vi.3}
    \mathfrak{P}(\lambda)\oplus I_{M-N}=\mathfrak{P}_{1}(\lambda)\cdot(\lambda I_{M}- \mathscr{L})\cdot \mathfrak{P}_{2}(\lambda) \quad \hbox{for all}\;\; \lambda\in\Omega_{\lambda_{0}}.
\end{equation}
Hence, bringing together \eqref{6.1} and \eqref{vi.3}, it becomes apparent that there are two finite dimensional curves $\mathfrak{E}_{1},\mathfrak{F}_{1}\in\mathcal{C}(\Omega_{\lambda_{0}},GL(\mathbb{C}^{M}))$ such that
\begin{equation}
\label{vi.4}
\mathscr{S}_{\mathfrak{L}(\lambda_{0}),(P,Q)}(\mathfrak{L}(\lambda)) \oplus I_{M-N}=\mathfrak{C}_{1}(\lambda)\cdot (\lambda I_{M}- \mathscr{L}) \cdot \mathfrak{F}_{1}(\lambda)
\quad\hbox{for all}\;\; \lambda\in\Omega_{\lambda_{0}}.
\end{equation}
Consequently, taking determinants on both sides of \eqref{vi.4} it is apparent that
\begin{align*}
\mathcal{D}[\mathscr{S}_{\mathfrak{L}(\lambda_{0}),(P,Q)}(\mathfrak{L}(\lambda))]&
= \mathcal{D}(\mathfrak{E}_{1}(\lambda)\cdot(\lambda I_{M}- \mathscr{L})\cdot \mathfrak{F}_{1}(\lambda))\\ 	 &=\mathcal{D}(\mathfrak{E}_{1}(\lambda))\cdot\mathcal{D}(\lambda I_{M}- \mathscr{L})\cdot \mathcal{D}(\mathfrak{F}_{1}(\lambda)).
\end{align*}
Thus, by Theorem \ref{th1.1},
\begin{equation*}
 \chi[\mathfrak{L},\lambda_{0}]=\ord_{\lambda_{0}} \mathcal{D}[\mathscr{S}_{\mathfrak{L}(\lambda_{0}),(P,Q)}(\mathfrak{L}(\lambda))]=
 \ord_{\lambda_{0}}\mathcal{D}(\lambda I_{M}- \mathscr{L}).
\end{equation*}
On the other hand, owing to Theorem \ref{th1.3}, we already know that
$$
  \mathcal{D}(\lambda I_{M}- \mathscr{L})
  =i(\mathcal{D}^{-1}(0),\lambda I_{M} - \mathscr{L};\lambda_{0} I_{M}-\mathscr{L}).
$$
Therefore,
$$
  \chi[\mathfrak{L},\lambda_{0}] = i(\mathcal{D}^{-1}(0),\lambda I_{M} - \mathscr{L};\lambda_{0} I_{M}-\mathscr{L}),
$$
which ends the proof of Theorem \ref{th1.4}.

\end{document}